\documentclass[11pt]{article}
\usepackage[utf8]{inputenc}  
\usepackage[british]{babel}
\usepackage{amsmath,amssymb,amsfonts,amsthm,bbm}
\usepackage{tikz,animate,media9}						
\usepackage{enumerate,mdframed}
\usepackage{graphicx,subfigure}

\usepackage{crossreftools,booktabs}
\usepackage{comment}

\makeatletter
\newcommand{\optionaldesc}[2]{%
  \phantomsection
  #1\protected@edef\@currentlabel{#1}\label{#2}%
}
\makeatother

\usepackage[mono=false]{libertine}
\usepackage[T1]{fontenc}
\usepackage[cmintegrals,libertine]{newtxmath}
\usepackage[cal=euler, calscaled=.95, frak=euler]{mathalfa}
\usepackage[a4paper,vmargin={2cm,2cm},hmargin={2.25cm,2.25cm}]{geometry}
\usepackage[font=sf, labelfont={sf,bf}, margin=1cm]{caption}
\usepackage[normalem]{ulem}
\usepackage{numprint}

\usepackage{hyperref}
\hypersetup{
	colorlinks=true,
		citecolor=blue!60!black,
		linkcolor=red!60!black,
		urlcolor=green!40!black,
		filecolor=yellow!50!black,
	breaklinks=true,
	pdfpagemode=UseNone,
	bookmarksopen=false,
}

\usepackage[capitalise]{cleveref}
\colorlet{link}{red!60!black}

\theoremstyle{plain}
\newtheorem{theorem}{Theorem}[section]
\newtheorem{lemma}[theorem]{Lemma}
\newtheorem{question}[theorem]{Question}
\newtheorem{claim}[theorem]{Claim}
\newtheorem{proposition}[theorem]{Proposition}

\newtheorem{corollary}[theorem]{Corollary}

\theoremstyle{definition}
\newtheorem{definition}[theorem]{Definition}

\theoremstyle{remark}
\newtheorem{remark}[theorem]{Remark}

\newcommand{\Var}[0]{\text{Var}}

\newcommand{\pr}[1]{\mathbb{P}\!\left(#1\right)}
\newcommand{\prstart}[2]{\mathbb{P}_{#2}\!\left(#1\right)}

\newcommand{\E}[1]{\mathbb{E}\!\left[#1\right]}

\newcommand{\prcond}[3]{\mathbb{P}_{#3}\!\left(#1\;\middle\vert\;#2\right)}
\newcommand{\econd}[2]{\mathbb{E}\!\left[#1\;\middle\vert\;#2\right]}

\newcommand{\Pb}[0]{\mathbb{P}}
\newcommand{\nkk}{n^{\frac{k}{k+1}}}
\newcommand{\nkkk}{n^{-\frac{k}{k+1}}}

\DeclareMathOperator{\LE}{LE}
\newcommand{\AB}{Aldous--Broder }

\DeclareMathOperator{\LERW}{LERW}

\DeclareMathOperator{\UST}{UST}

\newcommand{\cadlag}{c\`adl\`ag }

\renewcommand{\epsilon}{\varepsilon}

\newcommand{\T}{\mathcal{T}}
\newcommand{\Tbg}{\T_{\beta, \gamma}}
\newcommand{\TAB}{T^{\mathrm{AB}}}
\newcommand{\TABk}{T^{\mathrm{AB,k}}}
\newcommand{\TABn}{T^{\mathrm{AB}}_{\infty,n}}
\newcommand{\TABin}{T^{\mathrm{AB}}_{i,n}}
\newcommand{\TABjn}{T^{\mathrm{AB}}_{j,n}}
\newcommand{\TABjjn}{T^{\mathrm{AB}}_{j+1,n}}
\newcommand{\TABiin}{T^{\mathrm{AB}}_{i-1,n}}

\newcommand{\TWilk}{T^{\mathrm{Wil,k}}}

\newcommand{\branch}{\mathcal{B}}

\newcommand{\SB}{\textsf{SB}}

\newcommand{\SBk}{\textsf{SB}^{(j)}}

\newcommand{\eps}{\varepsilon}
\newcommand{\Av}{\textsf{Av}}
\newcommand{\Ter}{\textsf{Ter}}

\newcommand{\RR}{\mathbb{R}}

\newcommand{\XXbb}{\mathbb{X}}
\newcommand{\NN}{\mathbb{N}}

\author{
Eleanor Archer
\thanks{Modal'X, UMR CNRS 9023, UPL, Univ. Paris-Nanterre, F92000 Nanterre, France. \textsf{eleanor.archer@parisnanterre.fr}
} \qquad  
Matan Shalev 
\thanks{Tel Aviv University, Israel. \textsf{matanshalev@mail.tau.ac.il}
}
}
\title{Random choice spanning trees}

\begin{document}
\date{}
\vspace{-2cm}
\maketitle 

\begin{abstract}
    In this paper we introduce a new model of random spanning trees that we call \textit{choice spanning trees}, constructed from so-called \textit{choice random walks}. These are random walks for which each step is chosen from a subset of random options, according to some pre-defined rule. The choice spanning trees are constructed by running a choice modified version of Wilson's algorithm or the Aldous-Broder algorithm on the complete graph. We show that the scaling limits of these choice spanning trees are slight variants of random aggregation trees previously considered by Curien and Haas (2017). Moreover, we show that the loop-erasure of a choice random walk run on the complete graph converges after rescaling to a generalized Rayleigh process, extending a result of Evans, Pitman and Winter (2006). These are all natural extensions of similar results for uniform spanning trees.
\end{abstract}

\begin{figure}[h]

\includegraphics[width=16cm]{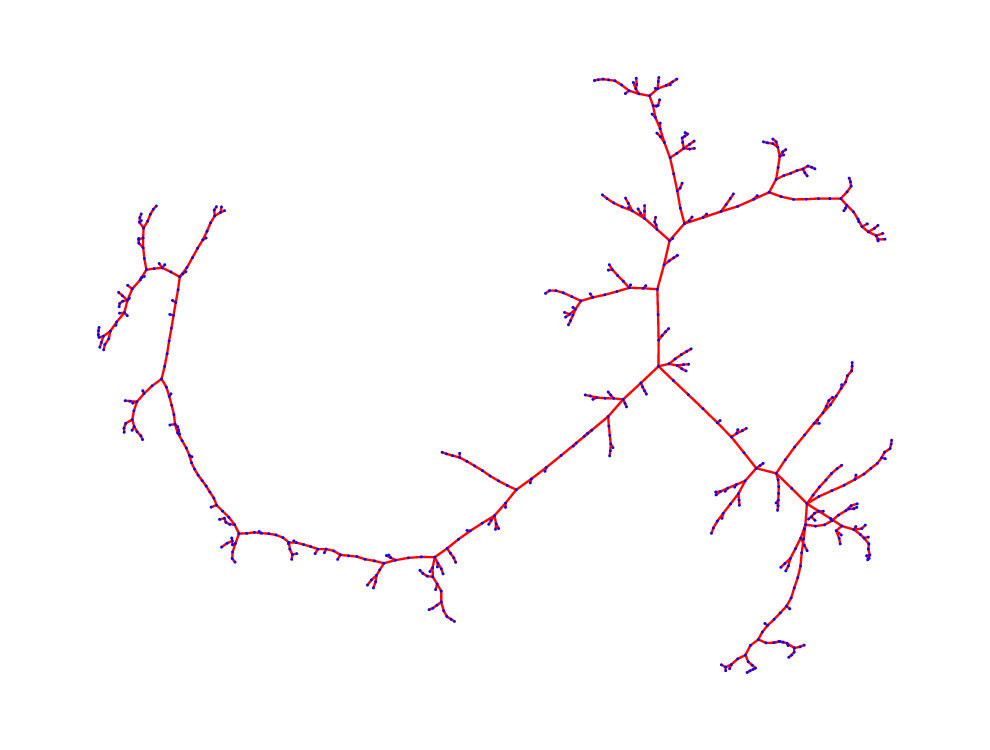}

\caption{A $3$-choice spanning tree}\label{fig:STs}
\end{figure}

\newpage
\tableofcontents

\section{Introduction}
Consider a uniformly labeled tree with $n$ vertices. A very classical result of Aldous \cite{AldousCRTI, AldousCRTII, AldousCRTIII} says that the scaling limit of this tree is the Brownian continuum random tree (CRT). In fact, such a discrete tree has the law of the \textbf{uniform spanning tree} (UST) of the complete graph $K_n$, and this result is much more general: the scaling limit of the uniform spanning tree of any sufficiently high-dimensional graph is also the CRT - see \cite{PeresRevelleUSTCRT, Schweinsberg, ANS2021ghp, archer2023ghp} and references therein.

The purpose of the present paper is to generalize Aldous' result in a different direction. It is well-known that standard algorithms for constructing USTs (such as Wilson's algorithm or the \AB algorithm) are essentially a discrete approximation of Aldous' \textbf{stick-breaking} construction of the CRT. In particular, the UST algorithms construct USTs branch-by-branch by removing loops from simple random walk trajectories. In a similar spirit, the stick-breaking construction involves sampling a collection of Poisson points on the positive half-line with intensity $t \ dt$, thus dividing it into a collection of ``sticks'', and then gluing these sticks together one at a time to construct the CRT branch-by-branch as well.

In this paper we introduce variants of these UST construction algorithms, in which we replace simple random walks with a different random walk. These variants output different random spanning trees of $K_n$ that we call \textbf{choice spanning trees} (CSTs). The key tool in the construction (and the reason for the nomenclature) of these CSTs is the use of the \textbf{choice random walk}. A $k$-choice random walk evolves like a simple random walk but with one key difference: at each step of the walk, we sample $k$ i.i.d. options for its next step, and then pick one of these options according to a pre-defined rule. They have been classically used in computer science to improve the efficacy of various algorithms; see \cite{avin2006power, georgakopoulos2020choice, georgakopoulos2022power} for more details.

In our case we choose a rule that (essentially) tries to avoid forming long loops. We will make the construction precise in Section \ref{sctn: crw}. This makes the branches of the spanning tree longer; we will soon see that for $k$-choice spanning tree of $K_n$, distances scale like $\nkk$ (we note that the case $k=1$ corresponds to the classical model of USTs).

It turns out that the scaling limits of these CSTs can be described by a stick-breaking process similar to that used to construct the CRT, but where we start with a Poisson process of intensity $t^k  dt$ at the beginning of the process. This model is included in those considered by Curien and Haas \cite{curien2017random} and proceeds as follows: the points in the Poisson process break the positive real line up into a series of intervals, that we call \textit{sticks}. We label the sticks in order according to their location on the real line, set $T^{(1)}$ to be equal to a single branch with length equal to that of the first stick, and at the $m^{th}$ step of the algorithm, we take the $m^{th}$ stick and attach it to a point chosen uniformly on $T^{(m-1)}$ (formally, according to renormalized Lebesgue measure on the union of its branches, denoted $\mu_m$). It turns out (see \cite{curien2017random}) that the closure of the limit of the pair $(T^{(m)}, \mu_m)$ exists and is compact almost surely; we denote it by $(\T_{k,0}, \nu_{k,0})$.

We will also need a slight generalization of this construction. We fix a parameter $\gamma \geq 0$ and before breaking the half-line into sticks, we endow it with a measure, denoted by $m_{\gamma}$, with density $t^{\gamma} dt$. We then construct the trees branch-by-branch using the sticks created from the Poisson process as before, but at each step the gluing location on $T^{(m-1)}$ is chosen according to the pushforward of $m_{\gamma}$, renormalized. We will show in Section \ref{sctn:SB convergence} that for appropriate values of $\gamma$, (the closure of) this tree and the associated measure similarly converge to a compact limit, which we denote by $(\T_{k,\gamma}, \nu_{\T_{k, \gamma}})$. In what follows, we will also let $d_{\T{k, \gamma}}$ denote the shortest distance metric on $\T_{k,\gamma}$ (we will see in Section \ref{sctn:SB convergence} that this is well-defined). 

The \AB algorithm (introduced by Aldous and Broder \cite{aldous1990random, broder1989generating}; both also acknowledge Diaconis) consists of running a random walk on a graph $G$ up until its cover time, and then setting $\TAB(G)$ to be equal to the set of first entry edges of each vertex. It can also be run in reverse \cite{HuReverse}. To form the $k$-choice \AB tree, we run a $k$-choice random walk on $G$ up until its cover time, then again set $\TABk (G)$ to be the collection of first entry edges. In fact we will see that for each $k$ there are two natural variants of the $k$ choice random walk to consider; we will call the two resulting spanning trees the \textbf{uniform $k$-choice spanning tree}, denoted $\TABk_{\text{unif}}(G)$, and the \textbf{maximal $k$-choice spanning tree}, denoted $\TABk_{\text{max}}(G)$. (The reasons for the nomenclature will become apparent when we give their precise definitions in Section \ref{sctn:AB alg def}.)

The main result for this algorithm is as follows. Here we let $d_n$ denote the graph metric on the appropriate CST, and $\mu_n$ denotes hitting measure on the CST for the choice random walk used to construct it (we will also formalize this in Section \ref{sctn:AB alg def}).

\begin{theorem}\label{thm:AB convergence}
For all $k \geq 1$, 
\begin{align*}
    \left(\TABk_{\text{max}}(K_n), n^{-\frac{k}{k+1}}d_n, \mu_n \right) &\overset{(d)}{\to} (\T_{k, k-1}, d_{\T_{k, k-1}}, \nu_{\T_{k, k-1}}) \\
    \left(\TABk_{\text{unif}}(K_n), n^{-\frac{k}{k+1}}d_n, \mu_n \right) &\overset{(d)}{\to} (\T_{k, 0}, d_{\T_{k, 0}}, \nu_{\T_{k, 0}})
\end{align*}
     with respect to the Gromov-Hausdorff-Prokhorov topology as $n \to \infty$. In addition, the Hausdorff dimension of $\T_{k, k-1}$ and $\T_{k, 0}$ is equal to $\frac{k+1}{k}$, almost surely.
\end{theorem}
The claim regarding the Hausdorff dimension was previously shown in \cite{curien2017random} for $\T_{k, 0}$. The Gromov-Hausdorff-Prokhorov (GHP) topology is a natural topology on metric spaces equipped with a measure; we will define it in Section \ref{sctn:GHP def}.

We can also play the same game with random spanning trees constructed using an analogous variant of Wilson's algorithm. In the classical setting, Wilson's algorithm starts by fixing an ordering of the vertices of a graph $G$, say $\{v_1, \ldots, v_n\}$, setting $T^{(0)}$ to be the single vertex $\{v_1\}$, and then at the $m^{th}$ step sampling a simple random walk from $v_m$ to $T^{(m-1)}$ and setting $T^{(m)}$ to be the union of $T^{(m-1)}$ with the loop-erasure of this new path. It was shown by Wilson \cite{Wil96} that this algorithm also outputs a sample of $\UST(G)$.

In our case, we will again run Wilson's algorithm but using $k$-choice random walks instead of simple random walks. Again there will be two natural options for the precise rule of the choice random walk, leading to the construction of two possible random spanning trees, denoted $\TWilk_{\text{unif}}(G)$ and $\TWilk_{\text{max}}(G)$.

\begin{theorem}\label{thm:choice trees the same}
For every $n \geq 1$ and $k \geq 1$,
\begin{align*}
\TWilk_{\text{max}}(K_n) \overset{(d)}{=} \TABk_{\text{max}}(K_n)\quad \text{ and } \quad \TWilk_{\text{unif}}(K_n) &\overset{(d)}{=} \TABk_{\text{unif}}(K_n).
\end{align*}
\end{theorem}

As a corollary we also of course recover the same scaling limits as in \cref{thm:AB convergence} for the Wilson choice trees, endowed with analogous metrics and measures.


Our final result concerns the building blocks of CSTs, the loop-erased choice random walks. In the case of simple random walks, it was shown by Evans, Pitman and Winter \cite{EPW} that if $(Y^{(n)}_m)_{m \geq 0}$ is a simple random walk on $K_n$, and $Z^{(n)}_m = |\LE(Y^{(n)}[0,m])|$, where $\LE$ denotes the operation of taking the loop-erasure, then
\begin{equation}\label{eqn:LERW Rayeigh intro}
\left(n^{-\frac{1}{2}}Z^{(n)}_{\lfloor m n^{\frac{1}{2}}\rfloor}\right)_{m \geq 0} \overset{(d)}{\to} (R_t)_{t \geq 0},
\end{equation}
where $(R_t)_{t \geq 0}$ is called the \textbf{Rayleigh process} and the convergence holds with respect to the Skorokhod-$J_1$ topology (this will also be defined in Section \ref{sctn:Skor J1 def}).

The analogous result is true in our setting. Fix $k \geq 1$; again we will see in Section \ref{sctn: crw} that there are two natural ways to define $k$-choice random walks on $K_n$. Denote these by $(Y^{(n,k,\text{unif})}_m)_{m \geq 0}$ and $(Y^{(n,k,\text{max})}_m)_{m \geq 0}$, and let $Z^{(n,k,\text{unif})}_m$ and $Z^{(n,k,\text{max})}_m$ denote the lengths of their corresponding loop-erasures.

\begin{theorem}\label{thm:Rayleigh conv}
For any $k \geq 1$,
    \begin{align*}
\left(n^{-\frac{k}{k+1}}Z^{(n,k,\text{max})}_{\lfloor m n^{\frac{k}{k+1}}\rfloor}\right)_{m \geq 0} \overset{(d)}{\to} (R^{(k)}_t)_{t \geq 0},
    \end{align*}
     with respect to the Skorokhod-$J_1$ topology as $n \to \infty$.
\end{theorem}

Here $R^{(k)}$ refers to a \textbf{generalized Rayleigh process}. We will define it precisely in Section \ref{sctn:gen Rayleigh def}, but its stationary measure satisfies
\[
\pr{X>t} = e^{-\frac{t^{k+1}}{k+1}}.
\]

The proof of Theorem \ref{thm:AB convergence} is similar to the proof of Aldous's convergence of discrete uniform trees to the CRT of \cite{AldousCRTI}, and to that of the convergence of aggregation trees in \cite{curien2017random}. In particular we first prove that the subtrees spanned by $i$ leaves converge to appropriate limiting subtrees, and then complete the argument using various tightness arguments. For the convergence of measures, this makes use of nice comparisons with urn models. The proof of the Hausdorff dimension of $\T_{k,k-1}$ also follows fairly standard arguments. Finally, Theorem \ref{thm:choice trees the same} follows by a direct calculation.

The proof of \cref{thm:Rayleigh conv} employs a representation of this process using Poisson point processes and also follows fairly straightforwardly by showing that the associated point processes converge.

We remark that throughout the paper we take the convention that the complete graph $K_n$ \textbf{has self-loops}.

\textbf{Acknowledgements.} We would like to thank Asaf Nachmias for helpful discussions. Both authors were supported by the ERC consolidator grant 101001124 (UniversalMap). EA was additionally supported by the ANR grant ProGraM (ANR-19-CE40-
0025). EA would like to thank Asaf Nachmias and MS for hosting her for a visit to Tel Aviv University during which time this project was started.

\textbf{Organization.} We start in Section \ref{sctn:preliminaries} by giving a recap of general UST constructions and explain how these can be adapted to the choice setting. In Section \ref{sctn:stick breaking defs} we consider the limiting trees $\T_{k,k-1}$ and show that they can be constructed as aggegration trees and have the claimed Hausdorff dimension. The rest of the paper concerns the choice random walks and choice trees. As a warm up, we start with the proof of \cref{thm:Rayleigh conv} in \cref{sec: lerw to rayleigh}, then treat the scaling limit of \cref{thm:AB convergence} in Section \ref{sec: AB trees}, and prove \cref{thm:choice trees the same} in \ref{sec: wilson trees}. We conclude in Section \cref{sctn:open questions} with some open problems on the model and some simulations.

\section{Preliminaries and constructions}\label{sctn:preliminaries}

\subsection{Gromov-Hausdorff-Prokhorov topology}\label{sctn:GHP def}

Here we define the GHP topology. We use the framework of \cite[Sections 1.3 and 6]{MiermontTessellations} and work in the space $\XXbb_c$ of equivalence classes of metric measure spaces (mm-spaces) $(X,d,\mu)$ such that $(X,d)$ is a compact metric space and $\mu$ is a Borel probability measure on it, and we say that $(X,d,\mu)$ and $(X',d',\mu')$ are equivalent if there exists a bijective isometry $\phi: X \to X'$ such that $\phi_* \mu = \mu'$ (here $\phi_*\mu$ is the pushforward measure of $\mu$ under $\phi$). To ease notation, we will represent an equivalence class in $\XXbb_c$ by a single element of that equivalence class. 
	
First recall that if $(X, d)$ is a metric space, the {\bf Hausdorff distance} $d_H$ between two sets $A, A' \subset X$ is defined as
	\[
	d_H(A, A') = \max \{ \sup_{a \in A} d(a, A'), \sup_{a' \in A'} d(a', A) \}.
	\]
For $\epsilon>0$ and $A\subset X$ we also let $A^{\epsilon} = \{ x \in X: d(x,A) < \epsilon \}$ be the $\epsilon$-fattening of $A$ in $X$. If $\mu$ and $\nu$ are two measures on $X$, the {\bf Prokhorov distance} between them is given by
	\[
	d_P(\mu, \nu) = \inf \{ \epsilon > 0: \mu(A) \leq \nu(A^{\epsilon}) + \epsilon \text{ and } \nu(A) \leq \mu(A^{\epsilon}) + \epsilon \text{ for any closed set } A \subset X \}.
	\]
	
	\begin{definition} \label{def:GHP} Let $(X,d,\mu)$ and $(X',d',\mu')$ be elements of $\XXbb_c$. The \textbf{Gromov-Hausdorff-Prokhorov} {\rm (GHP)} distance between $(X,d,\mu)$ and $(X',d',\mu')$ is defined as
		\[
		d_{\text{GHP}}((X,d,\mu),(X',d',\mu')) = \inf \left\{d_H (\phi(X), \phi'(X')) \vee d_P(\phi_* \mu, \phi_*' \mu') \right\},
		\]
		where the infimum is taken over all isometric embeddings $\phi: X \rightarrow F$, $\phi': X' \rightarrow F$ into some common metric space $F$.
	\end{definition}

\subsection{Skorokhod-$J_1$ topology}\label{sctn:Skor J1 def}

In this section we briefly introduce the Skorokhod-$J_1$ topology, first defined in \cite{Skorohod}, and used to give a notion of convergence for c\`adl\`ag functions that are not continuous.

The Skorokhod-$J_1$ distance function is defined as follows. First let 
\[
\Lambda = \{ \lambda:[0,1] \rightarrow [0,1] : \lambda(0)=0, \lambda(1)=1, \lambda \text{ a homeomorphism} \}.
\]
The Skorokhod-$J_1$ distance is then defined as
\begin{equation}\label{eqn:skor def}
d_{J_1}(f,g) = \inf_{\lambda \in \Lambda} \{ ||f\cdot \lambda - g ||_{\infty} + ||\lambda - I ||_{\infty}\}.
\end{equation}

By \cite[Theorem 13.1]{billingsley2013convergence}, we can prove Skorokhod convergence by first showing that the sequence of processes is tight, and then showing that the finite dimensional marginals converge for all $t\in [0,\infty)$.

For tightness, we will use the standard criteria in the Skorokhod-$J_1$ topology given in \cite[Theorem 13.2]{billingsley2013convergence}. If $X=(X_t)_{t \geq 0}$ is a \cadlag function, we define 
\[
||X||_{\infty} = \sup_{t \geq 0} |X_t|, \qquad \omega_\delta '(X) := \inf_{\{t_i\}_{i=0}^m} \sup_{1 \leq i \leq m} \sup_{s, t \in [t_{i-1}t_i)} |X_t - X_s|,
\]
where the infimum is taken over all partitions $0=t_0 < t_1 < \ldots < t_m$ with $t_i - t_{i-1} > \delta$ for all $i$. For $K\geq 0$, we also let $X^K$ denote the function defined by 
\[
X^K_t = X_t \mathbbm{1}\{t \leq K\}.
\]

We will use the following result.

\begin{proposition}\cite[Theorem 13.2]{billingsley2013convergence}.\label{prop:billingsley tightness criteria}
    Let $(X^{(n)})_{n \geq 0}$ be a sequence of \cadlag functions $[0, \infty) \to \RR$ defined on a probability space $(\Omega, \mathcal{F}, \Pb)$. For $K>0$ set $X^{(n,K)} = (X^{(n)})^K$. Then the sequence is tight if and only if for every $K>0$, the following conditions are satisfied.
    \begin{enumerate}[(a)]
        \item $\lim_{C \to \infty} \limsup_{n \to \infty} \pr{||X^{(n,K)}||_{\infty} \geq C} = 0$.
        \item For all $\eps>0$, $\lim_{\delta \downarrow 0} \limsup_{n \to \infty} \pr{ \omega_{\delta}'(X^{(n,K)}) > \eps} = 0$.
    \end{enumerate}
\end{proposition}

\subsection{Constructions of uniform spanning trees}

The uniform spanning tree of a finite graph $G$ (henceforth denoted $\UST (G)$) is a sample chosen uniformly at random from the set of spanning trees of $G$. We start by briefly recapping some well-known constructions of USTs from loop-erased random walks, and some of their associated properties.
	
\subsubsection{Loop-erased random walk}\label{sctn:LERW}
Wilson's algorithm \cite{Wil96}, which we now describe, is a widely used algorithm for sampling $\UST$s.
A {\bf walk} $X=(X_0, \ldots X_L)$ of length $L\in\mathbb{N}$ is a sequence of vertices where $(X_i, X_{i+1})\in E(G)$ for every $0 \leq i \leq L-1$. For an interval $J=[a,b]\subset[0,L]$ where $a,b$ are integers, we write $X[J]$ for $\{X_i\}_{i=a}^{b}$. Given a walk, we define its {\bf loop erasure} $Y = \LE(X) = \LE(X[0,L])$ inductively as follows. We set $Y_0 = X_0$ and let $\lambda_0 = 0$. Then, for every $i\geq 1$, we set $\lambda_i = 1+\max\{t \mid X_t = Y_{\lambda_{i-1}}\}$. If $\lambda_i \leq L$ we set $Y_i = X_{\lambda_i}$ and otherwise, we halt the process and set $\LE(X)=\{Y_k\}_	{k=0}^{i-1}$. The times $\left\{ \lambda_k(X)\right\}_{k=0}^{|\LE(X)| - 1}$ are the {\bf times contributing to the loop-erasure} of the walk $X$. When $X$ is a random walk starting at some vertex $v\in G$ and terminated when hitting a set of vertices $W$ ($L$ is now random), we say that $\LE(X)$ is the {\bf loop-erased random walk} ($\LERW$) from $v$ to $W$. Note that $\LE(X)$ is obtained by erasing the loops from $X$ in chronological order as they appear.

\subsubsection{Rayleigh process}\label{sctn:standard Rayleigh def}

In \cite{EPW}, the authors describe the \textit{Rayleigh process} as a process $(R_t)_{t\geq 0}$ with jumps that has the following dynamics. $(R_t)_{t\geq 0}$ grows deterministically with unit speed in between jumps, which occur with density $R_{s-}ds$ (where $R_{s-}$ is defined as the left limit of $R_{t}$ evaluated at $s$). When such a jump occurs, the process jumps to a uniform point in $[0,R_{t-}]$, and again grows at unit speed. 
More formally, this process can be constructed by sampling a homogeneous Poisson Point Process $\Pi$ on $\RR^+ \times \RR^+$ with Lebesgue intensity, and setting for each $t \geq 0$
\begin{equation*}
	R_t := \min\{t, \inf\{x+(t-s): (s,x) \in \Pi, s\leq t\}\}.
\end{equation*}
	This Rayleigh process is the Skorohod topology scaling limit of the process describing the length of a loop erased random walk on $K_n$ (see \cite{EPW,Schweinsberg}); more precisely, \eqref{eqn:LERW Rayeigh intro} holds.

\begin{figure}[!h]
\centering
\includegraphics[width=15cm]{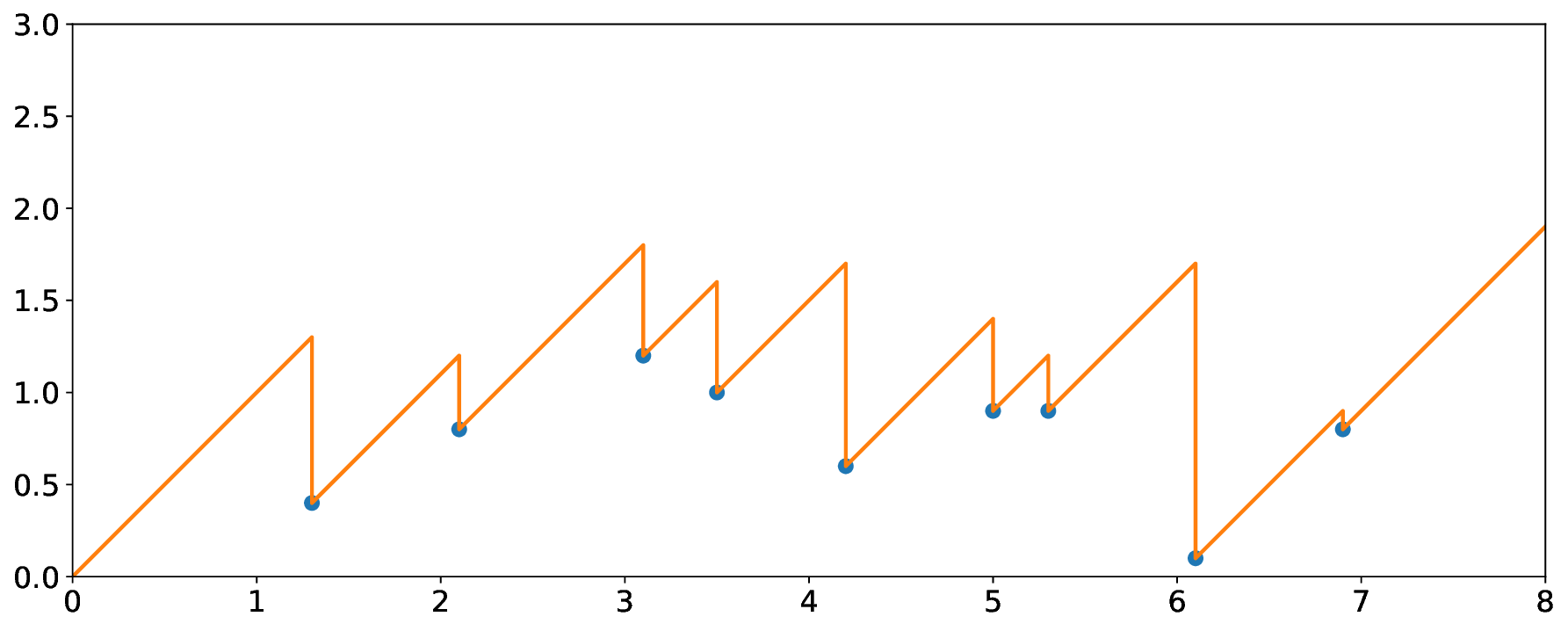}

\caption{Rayleigh process}\label{fig:Rayleigh standard}
\end{figure}

\subsubsection{Wilson's algorithm}
To sample a $\UST$ of a finite connected graph $G$ we begin by fixing an ordering of the vertices of $V=(v_1,\ldots, v_n)$. At the first step, let $T_0$ be the tree containing $v_1$ and no edges. At each step $i>1$, sample a $\LERW$ from $v_i$ to $T_{i-1}$ and set $T_i$ to be the union of $T_{i-1}$ and the $\LERW$ that has just been sampled. We terminate this algorithm with $T_n$. Wilson \cite{Wil96} proved that $T_n$ is distributed as $\UST(G)$, no matter the initial ordering of the vertices.

\subsubsection{Laplacian random walk}\label{sctn:Laplacian RW}
Here we outline the Laplacian random walk representation of the LERW (see \cite[Section 4.1]{LyonsPeres} for full details) and its application to Wilson's algorithm. Take a finite, weighted, connected graph $G$ and suppose we have sampled $T_j$ for some $j \geq 1$ using Wilson's algorithm as described above. We now sample a LERW from $v_{j+1}$ to $T_j$. Denote this LERW by $(Y_m)_{m \geq 0}$. Also let $X$ denote a random walk on $G$. For a set $A \subset G$, let $\tau_A$ denote the hitting time of $A$ by $X$, and $\tau_A^+$ denote the first return time to $A$ by $X$. The Laplacian random walk representation of $Y$ says that, conditionally on $T_j$ and on the event $\{(Y_m)_{m=0}^{i} \cap T_j = \emptyset\}$, we have for any $i \geq 0$ that 
\begin{align}\label{eqn:Laplacian RW classical formula}
\prcond{ Y_{i+1} = v}{(Y_m)_{m=0}^{i}}{} = \prcond{ X_{1} = v}{\tau_{T_j} < \tau^+_{\cup_{m=0}^{i} \{Y_m\}}}{Y_i} = \frac{\prstart{ X_{1} = v}{Y_i} \prstart{\tau_{T_j} < \tau_{\cup_{m=0}^{i} \{Y_m\}}}{v}}{\prstart{\tau_{T_j} < \tau^+_{\cup_{m=0}^{i} \{Y_m\}}}{Y_i}}. 
\end{align}
Clearly this is only non-zero when $v \notin \bigcup_{m=0}^{i} \{Y_m\}$.

\subsubsection{\AB algorithm}

The \AB algorithm was introduced separately by Aldous \cite{aldous1990random} and Broder \cite{broder1989generating} (both also acknowledge Persi Diaconis for useful discussions) as a method to sample USTs using random walks. It goes as follows. Start a random walk $X$ at any vertex of $G$ and run it until its cover time $\tau_{\text{cov}}$ (that is, the first time it has visited every vertex of $G$). Then let $T^{\text{AB}}(G)$ denote the set of first entrance edges to each vertex; that is
\begin{equation}\label{eqn:AB def}
T^{\text{AB}}(G)=\{(X_n, X_{n+1}): 1 \leq n < \tau_{\text{cov}} \text{ and } \nexists k \leq n \text{ such that } X_k = X_{n+1}\}.
\end{equation}
The remarkable fact established by Aldous and Broder is that the resulting tree $T$ has the distribution of $\UST(G)$.

We note the following: the \AB algorithm also constructs the UST ``branch-by-branch" in the sense that we can define a sequence of stopping times by $\sigma_0=0$ and, for $i \geq 1$,
\[
\sigma_{i+1}= \inf \{n > \sigma_i: \exists k \leq n \text{ such that } X_k = X_{n+1}\}.
\]
Note that the times $\tau_n$ correspond to the steps of the random walk $(X_n, X_{n+1})$ that are \textit{not} first entrance edges and therefore the times $(\sigma_i)_{i \geq 0}$ each correspond to an end of a branch in the resulting spanning tree.

For each $i \geq 0$, we can therefore define the partial spanning tree $\TAB_i = \TAB_i(G)$ by
\[
\TAB_i = \{(X_n, X_{n+1}): 1 \leq n < \sigma_i \text{ and } \nexists k \leq n \text{ such that } X_k = X_{n+1}\}.
\]
Note that the partial spanning trees form an increasing sequence and $\TAB_{i+1}$ differs from $\TAB_i$ by the addition of a single branch.

An interesting result of Hu, Lyons and Tang \cite{HuReverse} says that rather than taking the set of first entry edges of the random walk, we can instead take the set of last exit edges, and the corresponding subtree also has the law of $\UST(G)$.

\subsection{Constructions of choice spanning trees}
In this section we define two version of choice spanning trees. These are again sampled using a similar procedure as in the \AB algorithm and Wilson's algorithm, but we replace the random walk $X$ used in these procedures with a different stochastic process: more specifically the choice random walk that we will define in the next section.

\subsubsection{Choice random walk}\label{sctn: crw}

A $k$-choice random walk $(X_n)_{n=1}$ on a graph $G$ is a random walk in which, at time $n$, we sample $k$ random choices for the next step of the random walk and choose the one amongst them that optimizes our trajectory according to some strategy. See \cite{georgakopoulos2020choice}, \cite{georgakopoulos2022power} and references therein. In our case, our strategy is simply that we wish to avoid a portion of our past (in this paper, this portion will either be the entire path up until time $n-1$, or its loop-erasure), in additional to a pre-specified terminating set. More formally, we can define two variants (maximal and uniform) of this random walk formally as follows. First set $X_0=x$ for some $x \in G$, and choose a \textbf{terminating set} $\Ter \subset G$ (the walk will in fact be killed on hitting $\Ter$). Moreover we assume that $\Ter$ comes equipped with a \textit{partial order} $\preceq$. Then, given the trajectory $(X_0, X_1, \ldots, X_n)$ and a set $\Av_n \subset \{X_0, X_1, \ldots, X_n\}$, we define $X_{n+1}$ according to the following rules, in order of priority:
\begin{enumerate}
    \item Avoid the terminating set whenever possible.
    \item Avoid $\Av_n$ whenever possible, and if forced to hit $\Av_n$, hit it at the point which was either hit last by the random walk $X$ (\textbf{maximal} choice), or otherwise choose uniformly from the options in $\Av_n$ (\textbf{uniform} choice).
\end{enumerate}
(So if we have to choose between $\Ter$ and $\Av_n$, we prefer to hit $\Av_n$. $\Ter$ is also permitted to be empty.)

More formally, first sample $k$ independent vertices amongst the neighbours of $X_n$ and denote these $(Y^{(i)}_n)_{i=1}^k$. Let 
\begin{equation}\label{eqn:taun def}
\tau_n = \inf \{i \leq k: Y^{(i)}_n \notin \Ter \cup \Av_n\}, \qquad \tau_n (\Ter) = \inf \{i \leq k: Y^{(i)}_n \notin \Ter \},
\end{equation}
where we interpret the infimum of the empty set as infinity. On the event $\tau_n = \infty$, set 
\begin{align*}
N_n &=
     \arg \max \{i:\exists m \leq k \text{ such that } Y^{(m)}_n = X_i \text{ and } X_i \in \Av_n \text{ and } X_i \neq X_j \text{ for all } j>i\}.
\end{align*}
Note that $N_n$ is defined so that it corresponds to the choice of $Y_n^{(m)}$ that intersects the past and, among the possible choices, is the vertex that first appeared in the past at the latest time.

We define the \textbf{maximal choice} random walk by setting 
\begin{equation}\label{def:choice RW min}
X_{n+1} = \begin{cases} Y^{(\tau_n)}_n \text{ if }& \tau_n < \infty, \\
Y^{(N_n)}_n \text{ if }& \tau_n = \infty, \tau_n (\Ter) < \infty \\
\sup_{i\leq k} Y^{(i)}_n \text{ if }& \tau_n (\Ter) = \infty
\end{cases}
\end{equation}
(here the supremum in the final line is taken according to the partial order $\preceq$).

Also, we define the \textbf{uniform choice} random walk by setting $Z_0=x$ and 
\begin{equation}\label{def:choice RW uniform}
Z_{n+1} = \begin{cases} Y^{(\tau_n)}_n \text{ if }& \tau_n < \infty, \\
Y^{(\tau_n (\Ter))}_n \text{ if }& \tau_n = \infty, \tau_n (\Ter) < \infty \\
Y^{(1)}_n \text{ if }& \tau_n (\Ter) = \infty.
\end{cases}
\end{equation}

(Note its definition also depends on the rule for forming the sets $\Ter$ and $\Av_n$.) Note that all of these random walks share the property that they avoid making loops intersecting $\Av_n$ whenever possible, but in the case when they are forced to make loops, they choose a loop according to a different rule (respectively in order to maximize the length of the resulting trajectory, or to essentially choose uniformly from the options). We will continue to use the notation $X$ and $Z$ to distinguish between these two cases throughout the paper.

Note that these processes are no longer Markov. However that, since the loop-erasure defined in Section \ref{sctn:LERW} is a deterministic function of a path, it makes perfect sense to consider the loop-erasure of a choice random walk run up until any fixed time $n$, and therefore we can similarly define Wilson's algorithm and the \AB algorithm using these choice random walks. We do this in subsequent subsections.

\subsubsection{Generalized Rayleigh process}\label{sctn:gen Rayleigh def}

 
	Emulating the construction of Section \ref{sctn:standard Rayleigh def}, for every $k\geq 1$ we define the $k$-Rayleigh process (where the $1$-Rayleigh process is the process described in \cite{EPW}) as follows.
	
	Fix $k\geq 1$ and let $\Pi_k$ be a Poisson Point Process on $\RR^+ \times \RR^+$ with intensity measure $ky^{k-1}$. Then, define the $k$-Rayleigh process $(R_t^k)_{t\geq 0}$ by
	\begin{equation*}
		R_t^k := \min\{t, \inf\{x+(t-s): (s,x) \in \Pi, s\leq t\}\}.
	\end{equation*}
	In \cref{sec: lerw to rayleigh}, we will show that the process describing the length of the loop erasure of a $k$-choice random walk with the maximal choice law has the $k$-Rayleigh process as its scaling limit.

\begin{figure}[h]
\centering
\includegraphics[width=1\textwidth]{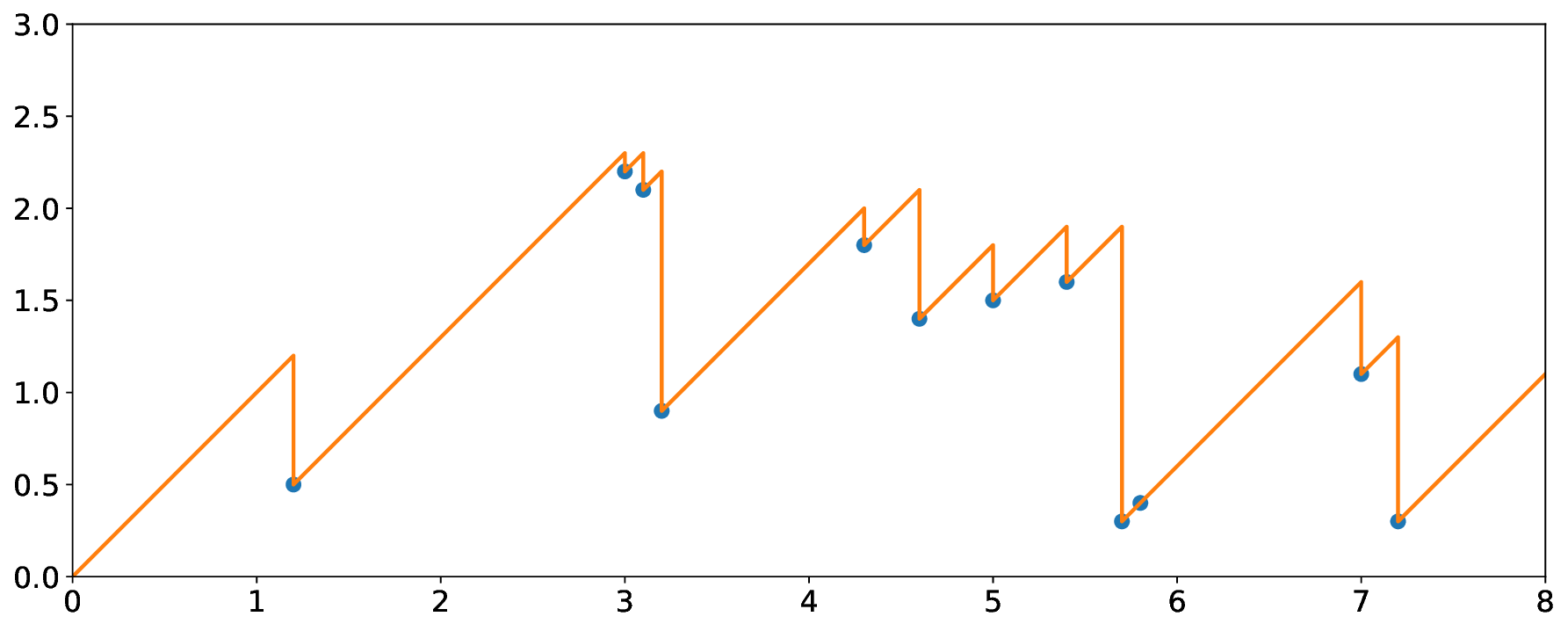}

\caption{$2$-Rayleigh process}\label{fig:Rayleigh choice}

\includegraphics[width=1\textwidth]{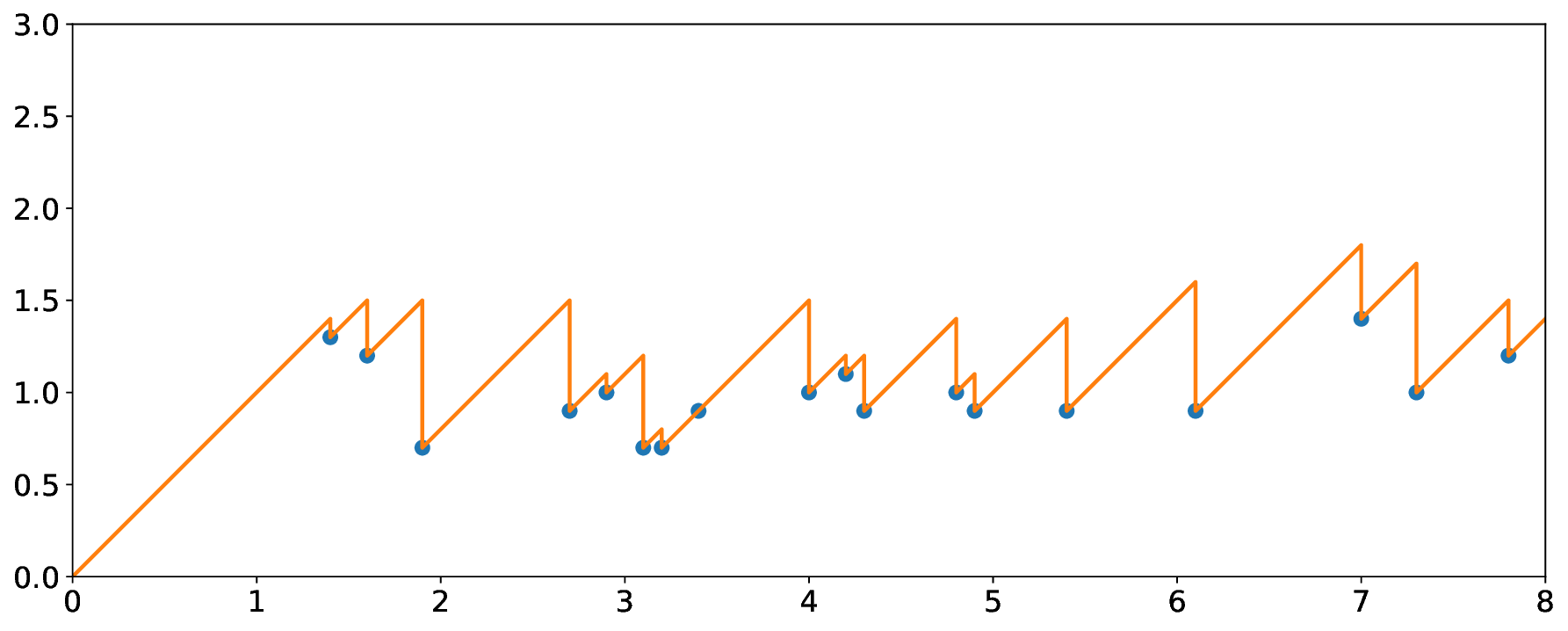}

\caption{$5$-Rayleigh process}\label{fig:Rayleigh choice}
\end{figure}

\subsubsection{Wilson's algorithm with choice}\label{sec: Wilson algorithm trees}

We define a maximal or uniform choice spanning tree using a variant of Wilson's algorithm on the complete graph as follows. Again, we begin by fixing $k \geq 1$ and an ordering of the vertices of $V=(v_1,\ldots, v_n)$. At the first step, let $T_0$ be the tree containing $v_1$ and no edges. Furthermore, say that $v_1$ has \textbf{time stamp} $(1,1)$, meaning that it is the first vertex added at the first step. At each step $i>1$, we again add another branch of the tree but this time we will also define a time stamp for each vertex in such a way that a time stamp of $(i,j)$ means that a vertex is the $j^{th}$ vertex appearing in the loop-erased process at step $i$ of Wilson's algorithm. Moreover, we can define an order on time stamps as follows: we say that 
\[
(i,j) \preceq (\ell,m) \text{ if and only if } i < \ell, \text{ or } i=\ell \text{ and } j \geq m.
\]
Note the slightly unexpected relation $j \geq m$ above (the reader may have expected $j \leq m$); this is to deal with the fact that branches in Wilson's algorithm are essentially added in reverse.


\textbf{Maximal choice Wilson spanning tree.} Conditionally on $T_i$ and the time stamps of all the vertices appearing in $T_i$, we sample a \textit{maximal} choice random walk (we emphasise that we \textit{always} use the maximal $k$-choice random walk to sample the branches, even when constructing the uniform choice spanning tree), starting from $X_0=v_{i+1}$, with $\Ter = T_i$, and $\Av_n = \{v: v \in \LE (X_0, \ldots, X_n)\}$, and the \textit{partial order of the time stamps}. We terminate the random walk when it hits $\Ter$. We have thus sampled a random path from $v_{i+1}$ to $T_i$; the tree $T_{i+1}$ is then defined as the union of the loop-erasure of this path, with $T_i$.

\textbf{Uniform choice Wilson spanning tree.} Conditionally on $T_i$ and the time stamps of all the vertices appearing in $T_i$, we sample a \textit{maximal} $k$-choice random walk (we emphasise that we \textit{always} use the maximal choice random walk to sample the branches, even when constructing the uniform choice spanning tree), starting from $X_0=v_{i+1}$, with $\Ter = T_i$, and $\Av_n = \{v: v \in \LE (X_0, \ldots, X_n)\}$, and a \textit{uniform partial order} on the vertices of $T_i$. This uniform partial order is resampled independently for each $i$ (i.e. each time we add a branch). We terminate the random walk when it hits $\Ter$. We have thus sampled a random path from $v_{i+1}$ to $T_i$; the tree $T_{i+1}$ is then defined as the union of the loop-erasure of this path, with $T_i$.


In both cases, we terminate the algorithm once $T_i$ spans the entire graph.  The final tree is called the maximal or uniform \textbf{$k$-choice Wilson spanning tree}, denoted $\TWilk_{\text{max}}(G)$ and $\TWilk_{\text{unif}}(G)$ respectively when the underlying graph is $G$.

We will in fact only consider this algorithm when $G$ is the complete graph for the following reason. A remarkable property of Wilson's original algorithm for uniform spanning trees is that the law of the final tree obtained does not depend on the ordering of the vertices used to initiate the algorithm. We do not expect this to be tree for choice spanning trees. However, on the complete graph the initial ordering clearly doesn't matter since everything is completely symmetric.

\subsubsection{\AB algorithm with choice}\label{sctn:AB alg def}
Given the definitions in \eqref{def:choice RW min} and \eqref{def:choice RW uniform}, it seems clear that the way to generalize the \AB algorithm on a graph $G$ is just to replace $X$ in \eqref{eqn:AB def} with either $X$ or $Z$ from \eqref{def:choice RW min} or \eqref{def:choice RW uniform} respectively. This is exactly what we do: we run the relevant choice random walk until its cover time, and then let $\TABk(G)$ be the collection of first entry edges of the choice random walk. We call the resulting structure the maximal or uniform \textbf{$k$-choice \AB spanning tree} of $G$, denoted $\TABk_{\text{max}}(G)$ and $\TABk_{\text{unif}}(G)$.

Note that, similarly to the classical \AB algorithm, we can define a sequence of stopping times by $\sigma_0=0$, and for each $i \geq 1$,
\begin{equation}\label{eqn:sigma i def}
\sigma_{i+1}= \inf \{m > \sigma_i: \tau_m = \infty\},
\end{equation}
(recall the definition of $\tau_m$ from \eqref{eqn:taun def}).
Again for each $i \geq 0$ we can use this to define the partial spanning tree
\begin{equation}\label{eqn:AB partial spanning tree}
\TABk_i = \{(X_m, X_{m+1}): 1 \leq m < \sigma_i \text{ and } \nexists k \leq m \text{ such that } X_k = X_{m+1}\}.
\end{equation}
Once again, the partial spanning trees form an increasing sequence and $\TABk_{i+1}$ differs from $\TABk_i$ by the addition of a single branch.

\subsection{Stick-breaking constructions of random trees}\label{sctn:stick breaking defs}

We start with a general description of how one can construct a sequence of deterministic trees from sticks on the real line.

\begin{definition}(Stick-breaking construction of a tree sequence). \label{def:stick breaking}
Set $y_0=z_0=0$, and suppose that we have a sequence of points $y_1, y_2, \ldots \in [0, \infty)$ and $z_1, z_2, \ldots \in [0, \infty)$ such that $y_{i-1} < y_{i}$ and $z_i \leq y_{i}$ for all $i \geq 1$. Construct trees as follows. Start by taking the line segment $[y_0, y_1]$ at time $1$. This is $T^{(1)}$ (as it contains one branch). We proceed inductively. At time $i \geq 2$, take the interval $(y_{i-1}, y_{i}]$ and attach the base of the interval $(y_{i-1}, y_{i}]
$ to the point on $T^{(i-1)}$ corresponding to $z_{i-1}$. This gives a new tree with $i$ branches, which we call $T^{(i)}$.

Given two such sequences and any $k \geq 2$ we define $\SBk ((y_0, y_1, y_2, \ldots), (z_0, z_1, z_2, \ldots ))$ or equivalently $\SBk ((y_0, y_1, y_2, \ldots, y_{j}), (z_0, z_1, z_2, \ldots, z_{j-1} ))$ to be equal to the tree $T^{(j)}$.
\end{definition}

In this paper, we will be interested in random trees built according to the following rules.

\begin{definition}\label{def:random stick breaking}
Take two parameters $\beta >0$ and $\gamma \in (0, \infty)$. Set $Y_0=Z_0=0$, let $(Y_1, Y_2, \ldots)$ denote the ordered set of points of a non-homogeneous Poisson process on $[0, \infty)$ with intensity $t^{\beta} \ dt$, and, conditionally on $Y_i$, let $Z_i$ be chosen on the interval $[0,Y_{i})$ according to the density $f^{\gamma}_{Y_i}(u) = \frac{(\gamma + 1)u^{\gamma}}{Y_{i}^{\gamma+1}}$ for each $i \geq 1$. Construct the sequence $(T^{(j)})_{j=1}^{\infty}$ as in Definition \ref{def:stick breaking}. Set $\Tbg$ to be the closure of the limit of $T^{(j)}$ in $\ell_1$ (when the limit exists - the embedding into $\ell^1$ is explained below).
\end{definition}

This law was first considered in the case $\beta=1$ and $\gamma=0$ by Aldous \cite[Section 4 and Process 3]{AldousCRTI} who showed that the limit is almost surely compact with Hausdorff dimension $2$, and moreover used this as the definition of the \textit{Brownian continuum random tree} (CRT). The model for general $\beta>0$ (but still $\gamma=0$) was later considered by Curien and Haas who showed in this case that the limiting tree is again compact and has Hausdorff dimension equal to $\frac{\beta +1}{\beta}$.

We can construct the sequence of partial trees $T^{(i)}$ as subspaces of $\ell_1$ as follows.  First sample the sequence $(y_1, y_2, \ldots)$ as in Definition \ref{def:random stick breaking}. The tree $T^{(0)}$ is equal to the point $(0,0,\ldots)$ and the tree $T^{(1)}$ is defined by
\[
T^{(1)} = (x,0,0,\ldots):x \leq y_1\}.
\] 
We then continue so that the $j^{th}$ branch $B_j$ (corresponding to a stick of length $y_j-y_{j-1}$) corresponds to the $j^{th}$ co-ordinate direction. To specify the embedding of a point $x \in T^{(i)}$ in $\ell_1$, take the path $P(x)$ from the root (labelled $0$) to $x$ in $T^{(i)}$, and let $x_j$ denote the length of the part of the path that intersects the $j^{th}$ branch (i.e. the length of the interval $B_j \cup P(x)$, measured according to Lebesgue measure; note that $x_j$ may be $0$). Then $x$ is mapped to $(x_1, x_2, \ldots)$. This embedding is clearly consistent in the sense that if $x \in T^{(i)}$, its embedding is the same in $T^{(I)}$ for all $I \geq i$, so this allows us to embed and then take closures in $\ell_1$.


We also define the projection $\rho : [0, \infty) \to \ell_1 $ in the natural way so that $\rho (t)$ denotes the image of $t$ under the stick-breaking construction of $\Tbg$. In particular, we can define this inductively by setting $\rho (0)=0$ and then if $t \in (y_{i-1}, y_{i}]$ we let $u=(u_1, u_2, \ldots, u_{i-1}, 0, 0, \ldots)$ be $\rho (z_{i-1})$, and set
\[
\rho (t) = (u_1, u_2, \ldots, u_{i-1}, t-y_i,  0, 0, \ldots)
\]
In other words, $\rho(t)$ is on the branch added to $T^{(i-1)}$ at step $i$, and is at distance $t-y_{i-1}$ from the grafting location. We also define the notion of the subtree formed at time $t$, generalizing slightly the definition in \eqref{eqn:AB partial spanning tree}. In particular, for each $t \geq 0$ we set 
\begin{equation*}\label{eqn:AB partial spanning tree time t}
\Tbg (t) = \rho ([0,t])
\end{equation*}
(so that $T^{(i)} = \Tbg (y_i)$).

This embedding also leads to the notion of descendance in $\Tbg$: we say that $x$ is a \textbf{descendant} of $y$ if $x$ if the representative of $x$ in $\ell_1$ is the concatenation of that of $y$ and another vector.

The following proposition will be useful for the comparison with spanning trees later on. It can be verified by a direct computation.

\begin{proposition}\label{prop:CRT stick probs}
Define the sequence $(Y_1, Y_2, \ldots)$ as in Definition \ref{def:random stick breaking}. Then for any $k \geq 1$ and any $x \geq 0$,
\[
\prcond{Y_{k+1} - Y_k \geq x}{(Y_i)_{i=0}^k}{} = \exp\left\{-\frac{1}{\beta+1} \left((Y_k+x)^{\beta+1} - Y_k^{\beta+1}\right)\right\}.
\]
\end{proposition}

We will need the following result, which we will prove in Section \ref{sctn:SB convergence}.

\begin{proposition}
For all $\beta > 0$ and $0 \leq \gamma \leq \beta - 1$, the limit $\Tbg$ almost surely exists and is compact. Moreover, there exists a measure $\mu$ on $\Tbg$ such that $(T^{(i)}, d_i, \mu_i) \to (\Tbg, d, \mu)$ almost surely with respect to the GHP topology, where $d_i$ and $d$ respectively denote the $\ell^1$ distance on the embedding of $T^{(i)}$ and $\Tbg$, and $\mu_i$ is the measure on $T^{(i)}$ defined via the renormalized projection of the function $f^{\gamma} = u^{\gamma}$ above.
\end{proposition}

\begin{remark}
In fact we will only use the case $\gamma=\beta-1$.
\end{remark}

We will use the following result from \cite{archer2023ghp} to compare trees formed from different stick-breaking constructions.

\begin{proposition}\cite[Proposition 3.7]{archer2023ghp}.\label{prop:stick breaking close}
Let $(y_0, y_1, y_2, \ldots), (z_0, z_1, z_2, \ldots)$ and $(y_0', y_1', y_2', \ldots), (z_0', z_1', z_2', \ldots)$ be the inputs to two separate stick-breaking processes as defined in Definition \ref{def:stick breaking}. Fix any $j \geq 1$ and let $T^{(j)}$ and $T^{(j)'}$ be the trees formed after $j$ steps of the processes (i.e. each with $j$ branches). Let $d$ and $d'$ denote distances on $T^{(j)}$ and $T^{(j)'}$.

Fix some $\epsilon>0$ and suppose that the following holds.
\begin{enumerate}[(i)]
    \item $|y_i - y_i'| \leq \epsilon$ for all $i \leq j$ and $|z_i - z_i'| \leq \epsilon$ for all $i \leq j-1$,
    \item $|z_i - y_{\ell}| \geq 3\epsilon$ for all $i \leq j-1,\ell \leq j$.
\end{enumerate}
Then, for all $0 \leq i,\ell \leq j$, it holds that
\[
|d(y_i, y_{\ell}) - d'(y_i', y_{\ell}')| \leq 2j\epsilon.
\]
\end{proposition}

\subsection{Random variables}
Here we present two elementary results that will be useful in later sections.

\begin{claim}\cite[Claim 2.9]{archer2023ghp}.\label{cl: two uniforms are close}
Let $\eps>0$ and let $0<a<b$ with $b-a \leq \eps$. Let $X_a$ be distributed as $\textsf{Uniform}\left(\left[0,a\right]\right)$ and $X_b$ as $\textsf{Uniform} \left(\left[0,b\right]\right)$. Then, we can couple $X_a$ and $X_b$ such that $\pr{|X_a - X_b| > \eps} = 0$.
\end{claim}

	\begin{lemma}\label{lem: proh close rv}
	For any $L>0$, $k \in \NN$, let $X_{L,k}$ be the random variable on $(0, \infty)$ satisfying 
	\[
	\pr{X_{L,k} > x} = \exp\left\{-\frac{(x+L)^{\frac{k+1}{k}}-L^{\frac{k+1}{k}}}{k+1}\right\}.
	\]
	Then for any $\delta > 0$, there exists $\eta = \eta (\delta, L, k)>0$ such that the following holds. Let $Y$ be another random variable on $(0, \infty)$, and suppose that for all $x>0$,
	\begin{equation}\label{eqn:variables close}
	|\pr{X_{L,k} > x} - \pr{Y > x}|<\eta.
	\end{equation}
	Then this implies that we can couple $X_{L,k}$ and $Y$ so that $\pr{|X_{L,k} - Y|>\delta} < \delta$.
	
	Furthermore, for any $\delta, k, L_1$ and $L_2$ with $L_1< L_2$, there exists $\eta = \eta(\delta, L_1, L_2,k)$ such that we can couple $X_{L,k}$ and $Y$ as described above for every $L\in [L_1,L_2]$.
	\end{lemma}
\begin{proof}
    This was proved in \cite[Lemma 2.10]{archer2023ghp} in the case $k=1$. The proof for general $k$ is the same.
\end{proof}

\subsection{Urn models}\label{sctn:urn}

Here we briefly recap some background on urn models, following the presentation of \cite[Section 4, page 15]{AldousCRTI}.

We let $U_m$ and $V_m$ respectively denote the number of black and white balls in an urn at time $m$. At time $m+1$, we choose a ball uniformly among those present, and then add $\Delta_{m+1}$ balls of the same colour, where $\Delta_{m+1}$ is independent of the colour chosen. Letting $A_m$ denote the event that the colour chosen is black, we therefore have
\begin{align*}
    (U_{m+1}, V_{m+1}) = \begin{cases}
    (U_{m} + \Delta_{m+1} , V_{m+1}) \text{ on }& A_m, \\
    (U_{m} , V_{m+1} + \Delta_{m+1}) \text{ on }& A_m^c,
    \end{cases}
    \qquad \ \prcond{A_m}{(U_m, V_m)}{} = \frac{U_m}{U_m+V_m},
\end{align*}
and, given $U_m$ and $V_m$, the variables $A_m$ and $\Delta_{m+1}$ are conditionally independent. We emphasise that $U_m, V_m$ and $\Delta_m$ are not required to be integers. Then (see \cite[Equations (25) and (26)]{AldousCRTI}),
\begin{align}\label{eqn:urn variance bound}
\begin{split}
    R_m = \frac{U_m}{U_m + V_m} \text{ is a martingale,} \\
    \Var ({R_{m+1}}) - \Var ({R_m}) \leq \E{ \left( \frac{\Delta_{m+1}}{U_m + V_m} \right)^2}.
\end{split}
\end{align}

\section{Convergence of stick-breaking with general gluing rules}\label{sctn:SB convergence}

In the case $\gamma=0$, it was shown in \cite[Theorem 1.1]{curien2017random} that the limit $\Tbg$ exists, is compact and has Hausdorff dimension equal to $\frac{\beta+1}{\beta}$ almost surely. Our aim in this section is to extend this to all $0 \leq \gamma \leq \beta -1$. Note that we will only need the case $\beta-1=\gamma$.

\subsection{Convergence and compactness}\label{sctn:SB conv and comp}

Fix $\beta, \gamma\geq 0$ and recall the definition of the partial trees $(T^{(i)})_{i \geq 1}$ from Section \ref{sctn:stick breaking defs}. 
%
We follow the strategy of \cite[Section 4]{AldousCRTI} (for the case $\beta=1, \gamma=0$) and proceed via a series of lemmas. For these we will need a few definitions. We also recall the definition of the projection $\rho$ defined in \eqref{eqn:AB partial spanning tree time t}.

Also let $d$ denote the graph metric on the tree $\Tbg$. For the next few propositions, we define $D$ by 
\[
D(s,t) = \inf_{0 \leq r \leq s} d(\rho(r), \rho (t)) \qquad 0<s<t.
\]

\begin{proposition}\label{prop:tightness stick breaking exp tail}(cf \cite[Lemma 5 and Lemma 9]{AldousCRTI}).
Suppose that $0 \leq \gamma \leq \beta -1$. For all $t>s\geq 1$ and all $x>0$,
\[
\pr{D (s,t) \geq x} \leq e^{-\frac{x}{2}s^{\gamma+1}}.
\]
\end{proposition}
\begin{proof}
We follow an adaptation of the proof of \cite[Lemma 9]{AldousCRTI}. In particular, fix $t>s\geq 1$ and then take $\eps>0$ small enough that $\eps(k+1)t^{\gamma+1} \ll 1$. We then claim that for all $r \in [s,t]$, the quantity $D(s,t)$ is stochastically dominated by $\eps X$ where $X \sim \textsf{Geometric}(\eps s^{\gamma+1})$. Note that this is automatically true for all $r \in [s, s+\eps]$; for $r \in [s+\eps, t]$ we proceed by induction and assume that $D(s,r)$ is dominated by the claimed random variable for all $r \in [s, s+i\eps]$ for some $i \geq 1$.

Now suppose that $q \in [0,1]$, that $A,B$ and $C$ are three events partitioning the probability space, that $\pr{A}\geq q$, and $\psi_r$ is a random variable satisfying
\begin{equation}\label{eqn:domination Aldous geometric}
\psi_r \leq \begin{cases} \eps &\text{ on the event } A, \\
     \eps + \psi_{r-\eps} &\text{ on the event } B, \\
      \eps + \eta_{r-\eps} &\text{ on the event } C,
\end{cases}
\end{equation}
where $\psi_{r-\eps}$ and $\eta_{r-\eps}$ are both dominated by \textsf{Geometric}($q$) random variables and are independent of $A$, $B$ and $C$. Then it follows by \cite[Equation (37)]{AldousCRTI} that if $\psi_{r-\eps}$ and $\eta_{r-\eps}$ are dominated by $\eps$ times a \textsf{Geometric}($q$) random variable, then so is $\psi_r$.

We take $q=\frac{1}{2}\eps s^{\gamma+1}$, $\psi_r=D(s,r)$ and apply this to the events
\begin{align*}
    A &= \{\exists i \text{ such that } Y_i \in [r-\eps, r] \text{ and } Z_i \in [0,s] \text{ for the maximal such } i \} \\
    B &= \{\nexists i \text{ such that } Y_i \in [r-\eps, r], \text{ or } \exists i \text{ such that } Y_i \in [r-\eps, r] \text{ and } Z_i \in [r-\eps,r]  \text{ for the maximal such } i\} \\
    C &= \{\exists i \text{ such that } Y_i \in [r-\eps, r] \text{ and } Z_i \in [s,r-\eps] \text{ for the maximal such } i \}.
\end{align*}
It then follows by an easy calculation that \eqref{eqn:domination Aldous geometric} is satisfied (recall that $\gamma +1 \leq \beta$):
\[
\pr{A} \geq \pr{\textsf{Poisson}(\eps(r-\eps)^{\beta}) \neq 0}\left(\frac{s}{r}\right)^{\gamma+1} \geq \frac{1}{2}\eps s^{\gamma+1}.
\]
Moreover, the required independence assumptions hold as a result of the Poisson structure, and on the event $C$ we take $\eta_{r-\eps}$ to be a weighted average over the laws of $\psi_{r'}$ for $r' \in [s,r-\eps]$; these all satisfy the required domination by the inductive hypothesis (as does $\psi_{r-\eps}$ for the event $B$). Since $t$ was arbitrary and we can always choose a sufficiently small $\eps>0$ to run this argument, this establishes the claimed domination.

To conclude, note that it therefore follows that
\begin{align*}
 \pr{D (s,t) \geq x} \leq \pr{X \geq \eps^{-1} x} \leq \left(1-\frac{1}{2}\eps s^{\gamma+1}\right)^{\eps^{-1} x} \leq \exp \{-\frac{1}{2}s^{\gamma+1}x\}.
\end{align*}
\end{proof}
This leads us to the following.

\begin{corollary}(cf \cite[Lemma 6]{AldousCRTI}).\label{cor:Aldous sup D interval SB}
Suppose that $0 \leq \gamma \leq \beta -1$. Then, for any $j \geq 0$,
    \[
    \pr{\sup_{e^j \leq t \leq e^{j+1}} D (e^j, t) > 3(\beta+2)je^{-j({\gamma+1})}} \leq e^{2\beta+3-j}
    \]
\end{corollary}
\begin{proof}
Set $s=e^j, u=e^{j+1}$ and
\[
\tau = \tau(x) = \inf\{t \geq s: D(s,t) \geq x\}.
\]
We have by Proposition \ref{prop:tightness  stick breaking exp tail} that, for any $s<u$ and any $x>0$:
    \begin{align}\label{eqn:Aldous exp bound interval UB SB}
        \E{\textsf{Leb} \{s \leq t \leq u+1: D(s,t) \geq x\}} \leq (u+1-s)e^{-\frac{x}{3}s^{\gamma+1}}.
    \end{align}
Now for any $r \geq 0$ set $C_{r} = \inf\{r' > r: \exists i \text{ such that } Y_i=r'\}$. Then
\begin{align*}
\{t \in [\tau, C_{\tau}]\} \subset \{D(s,t) \geq x\},
\end{align*}
and hence
\begin{align}\label{eqn:Aldous exp bound interval LB SB}
\begin{split}
\E{\textsf{Leb} \{s \leq t \leq u+1: D(s,t) \geq x\}} &\geq \E{\mathbbm{1}\{\tau \leq u+1\} (C_{\tau} \wedge (u+1)-\tau)} \\
 &\geq \pr{\tau \leq u+1}\econd{C_{\tau} \wedge (u+1)-\tau}{\tau \leq u+1}.
 \end{split}
\end{align}
Now note that, for any $t \geq 0$, 
\begin{align*}
    \E{(C_t \wedge (u+1) - t)} \geq \int_0^{u+1} \pr{ (C_t \wedge (u+1) - t) >r} dr &\geq \int_0^{u+1} \pr{\textsf{Poi}(r(u+1)^{\beta})=0} dr \\
    &\geq  (u+1)^{-\beta}(1-e^{-(u+1)^{\beta+1}}).
\end{align*}
By the strong Markov property, we can substitute this into \eqref{eqn:Aldous exp bound interval LB SB} and then compare to \eqref{eqn:Aldous exp bound interval UB SB} to deduce that
\begin{align*}
    \pr{\tau \leq u+1} \leq \frac{(u+1-s)e^{-\frac{x}{3}s^{\gamma+1}}}{(u+1)^{-\beta}(1-e^{-(u+1)^{\beta+1}})} \leq e^{2\beta+2} s^{\beta+1} e^{-\frac{x}{3}s^{\gamma+1}}.
\end{align*}
In particular, reintroducing the dependence on $x$ and taking $x=\frac{3(\beta+2)\log s}{s^{\gamma+1}}$ gives
\begin{align*}
    \pr{\tau\left(\frac{3(\beta+2)\log s}{s^{\gamma+1}}\right) \leq u+1} \leq e^{2\beta+3} s^{-1} = e^{2\beta+3-j}.
\end{align*}
\end{proof}
This allows us to tie up the proof of Hausdorff convergence.

\begin{corollary}\label{cor:stick breaking conv}
Suppose that $0 \leq \gamma \leq \beta -1$. Almost surely,
\[
\lim_{i \to \infty} (d_{H}(T^{(i)}, \Tbg) = 0.
\]
\end{corollary}
\begin{proof}
By summing over $j \geq J$ and applying a union bound we deduce from Corollary \ref{cor:Aldous sup D interval SB} that
 \begin{equation}\label{eqn:SB D total distance}
    \pr{\sup_{e^J \leq t} D_n (e^J, t) > 6(\beta+2)Je^{-J(\gamma+1)}} \leq \pr{\exists j \geq J: \sup_{e^j \leq t \leq e^{j+1}} D_n (e^j, t) > 3(\beta+2)je^{-j({\gamma+1})}} \leq e^{2\beta+4-J}.
 \end{equation}
In particular this implies that $ \pr{d_{H}(\Tbg(e^J), \Tbg) \geq 6(\beta+2)Je^{-J(\gamma+1)}}$ enjoys the same upper bound, and hence by Borel-Cantelli we deduce that $d_{GH}(\Tbg(e^J), \Tbg) \to 0$ almost surely. The result follows by noting that $Y_i \to \infty$ almost surely.
\end{proof}

The previous corollary also shows that $\Tbg$ is compact (since the space of non-empty compact subspaces of $\ell^1(\RR)$ endowed with the Hausdorff metric is complete.

\subsection{Convergence of measures}

For the case $\gamma=0$, it was shown in \cite[Theorem 2.5]{curien2017random} that the renormalized length measures converge using martingale techniques. This is in the same spirit as the argument used in \cite[Theorem 3]{AldousCRTI} for the case of the CRT.

More generally, we should be able to apply the same argument to the measure defined by renormalizing the measure with density $f^{\gamma}(u) := u^{\gamma}$. In particular, for each $i \geq 1$ recall that $Z_i$ was chosen on the interval $[0,Y_{i})$ according to the density $f_{Y_i}(u) = \frac{(\gamma + 1)u^{\gamma}}{Y_{i}^{\gamma+1}}$. We let $\mu_i$ denote the corresponding pushforward measure onto $T^{(i)}$, i.e. defined by $\mu_i(A) = \pr{Z_i \in \rho^{-1}(A)}$ where $\rho$ is the canonical projection defined at the start of Section \ref{sctn:SB conv and comp}.

For a set $A \subset \Tbg$, we let $A^{\uparrow}$ denote the set of descendants of $A$ (see Section \ref{sctn:stick breaking defs} for the definition of descendant). The key observation leading to convergence of the measures $\mu_i$ is that for any $A \subset T^{(i)}$, the sequence $(\mu_j (A^{\uparrow}))_{j \geq i}$ is a martingale. Indeed, if $a_j = \int_{Y_{j-1}}^{Y_j} (\gamma + 1)u^{\gamma} du$ and $A_j= \int_{0}^{Y_j} (\gamma + 1)u^{\gamma} du$, then
\begin{align}\label{eqn:measure MG prop}
    \econd{\mu_{j+1}(A^{\uparrow})}{\mu_{j}(A^{\uparrow})} = \mu_j(A^{\uparrow}) \frac{A_j\mu_j(A^{\uparrow})+a_{j+1}}{A_{j+1}} + \left( 1 - \mu_j(A^{\uparrow}) \right) \frac{A_j\mu_j(A^{\uparrow})}{A_{j+1}} 
    = \mu_j(A^{\uparrow}).
\end{align}
Since $\mu_{j}(A^{\uparrow}) \in [0,1]$, it therefore follows that $\mu_{j}(A^{\uparrow})$ converges almost surely. This leads to the following proposition.

\begin{proposition}(cf \cite[Theorem 3]{AldousCRTI}, \cite[Lemma 2.9]{curien2017random}).\label{prop:measure conv SB}
There exists a measure $\mu$ on $\Tbg$ such that $\mu_j \to \mu$ almost surely as $j \to \infty$. Moreover, $\mu (\mathcal{L}(\Tbg)) = \mu (\Tbg)=1$, and $\text{supp} (\mu) = \Tbg$.
\end{proposition}
\begin{proof}
We claim that $(\mu_j)_{j \geq 1}$ is a Cauchy sequence with respect to the Prokhorov metric. This follows from \eqref{eqn:measure MG prop} by exactly the same proof as \cite[Lemma 2.9]{curien2017random}, with Corollary \ref{cor:stick breaking conv} playing the role of \cite[Lemma 2.8]{curien2017random}. Since the space of measures on a compact metric space is complete, this guarantees existence of a limit measure $\mu$, and moreover implies that $\mu (\Tbg)$ is necessarily $1$ (because $\mu_j (\Tbg)=1$ for all $j$).

To show that $\mu (\mathcal{L}(\Tbg))=1$, it is therefore equivalent to show that $\mu (\Tbg \setminus \mathcal{L}(\Tbg))=0$, which is implied by
\begin{align}
\lim_{j \to \infty}  \mu_j (T^{(i)}) = 0 \qquad \forall i \geq 1.
\end{align}
This is clearly the case since $\mu_j(T^{(i)}) = \frac{Y_i^{\gamma}}{Y_j^{\gamma}} \to 0$ as $j\to \infty$.

For the final result regarding the support of $\mu$, we claim that the sequence $(Z_i)_{i \geq 1}$ is dense in $\RR^+$. To prove this, first note that, by Borel-Cantelli, the number of $Y_i$ falling in the interval $[2^m, 2^{m+1}]$ is at least $2^{m(1/2+\beta)}$ for all sufficiently large $m$, almost surely. Suppose this holds for all $m \geq M$. Now take an interval $I_{x,\eps} := [x, x+\eps] \subset \RR^+$, and take $m \geq m_x := {\lfloor \log_2 (x) \rfloor} \vee M$. Then the number of Poisson points in the interval $[2^m, 2^{m+1}]$ is at least $2^{m(1/2+\beta)}$, and hence the number of corresponding $Z_i$ falling in $I_{x,\eps}$ stochastically dominates a \textsf{Binomial}($2^{m(1/2+\beta)}, 2^{-\beta m} c_{\eps, x}$) random variable, where $c_{\eps, x} \in (0, \infty)$ is a fixed constant. Hence this number is at least $2^{m/4}$, eventually almost surely, and hence the total number of $Z_i$ falling in this interval is infinite (in fact we just need that it is strictly positive). This holds simultaneously for all rational $x$ and $\eps$ hence it follows that the sequence $Z_i$ is dense in $\RR^+$. It follows that the projected points are therefore dense in $\Tbg$. It can then be argued exactly as in the proof of \cite[Theorem 3]{AldousCRTI} that these projected points are all in $\text{supp} (\mu)$, and hence $\text{supp} (\mu) = \Tbg$.
\end{proof}

See also \cite[Section 4]{AldousCRTI} and \cite[Section 2.3]{curien2017random} for similar arguments for the case $\gamma=0$.

\subsection{Hausdorff dimension}
We now turn to the Hausdorff dimension. In the case $\gamma=0$ this was established to be equal to $\frac{\beta+1}{\beta}$ in \cite{curien2017random}.

\begin{proposition}
    Suppose that $0 \leq \gamma \leq \beta -1$. Then, almost surely,
    \[
    \frac{\beta+1}{\beta} \leq \dim_H (\Tbg) \leq \frac{(\gamma+2) \vee \beta}{\gamma +1}.
    \]
    In particular, when $\beta=\gamma+1=k$, then $\dim_H(\Tbg) = \frac{k+1}{k}$ almost surely.
\end{proposition}
\begin{proof}
For $\delta \geq 0$ set $s(\delta) = \delta^{-\frac{1}{\gamma +1}} (\log \delta^{-1})^{\frac{1}{\gamma +1}}$. Note that, setting $\delta= 6(\beta+2)Je^{-J(\gamma+1)}$, it follows from \eqref{eqn:SB D total distance} and Borel-Cantelli that
 \begin{equation*}
    \pr{\sup_{s(\delta) \leq t} D (s(\delta), t) > \delta \text{ i.o. }}=0.
 \end{equation*}
 Whenever this supremum is indeed less than $\delta$, any $\delta$-covering of the projection $\rho ([0, s(\delta)])$ is a $2\delta$-covering of $\Tbg$. A $\delta$-covering of the former set is given by 
 \[
 (\rho (i \delta))_{0 \leq i \leq \delta^{-1} \cdot s(\delta))} \cup \{\rho (Y_i^+): Y_i \leq s(\delta)\}.
 \]
The first set in the union has size $\delta^{-\frac{\gamma + 2}{\gamma +1}} (\log \delta^{-1})^{\frac{1}{\gamma +1}}$. The size of the second set has law \textsf{Poisson}($\frac{1}{\beta+1}s(\delta)^{\beta+1})$ so is upper bounded by $ 10\delta^{-\frac{\beta+1}{\gamma +1}} (\log \delta^{-1})^{\frac{1}{\gamma +1}}$ eventually almost surely (by applying Borel-Cantelli along the subsequence $\delta_m = 2^{-m}$ and using monotonicity). This implies that the Hausdorff dimension of $\Tbg$ is at most $\frac{(\gamma+2) \vee (\beta+1)}{\gamma +1}$ almost surely.

For a lower bound, we lower bound the packing number using the argument of \cite[Lemma 7]{AldousCRTI}. For this, we let $A(t, \delta)$ denote the event that there is no $i$ such that $Y_i \in [t-2\delta, t]$, and let
\begin{align*}
M_{\delta} = \{\rho (2j \delta): j \geq 1, A(2j \delta, \delta) \text{ occurs}\}.
\end{align*}
Then $M_{\delta}$ is a $\delta$-packing of $\Tbg$. Moreover, by independence, there exist $C, C'<\infty$ and $c>0$ such that
\begin{align*}
\Var{|M_{\delta}|} \leq C \int_0^{\infty}  \delta^{-1} \exp \{-\delta t^{\beta}\} dt = C \int_0^{\infty} \delta^{-1-\frac{1}{\beta}} \exp \{-t^{\beta}\} dt \leq C'\delta^{-\frac{\beta+1}{\beta}}, \qquad \E{|M_{\delta}|} \geq c\delta^{-\frac{\beta+1}{\beta}}
\end{align*}
(the calculation for the expectation is identical). Hence it follows from Chebyshev's inequality that there exists $C''<\infty$ such that
\begin{align*}
    \pr{|M_{\delta}| \leq \frac{c}{2}\delta^{-\frac{\beta+1}{\beta}}} \leq C''\delta^{\frac{\beta+1}{\beta}},
\end{align*}
and hence by applying Borel-Cantelli along the subsequence $\delta_m = 2^{-m}$ and using monotonicity we deduce that there exists $c'>0$ such that $M_{\delta} \geq c'\delta^{-\frac{\beta+1}{\beta}}$ eventually almost surely, and therefore that the Hausdorff dimension is lower bounded by $\frac{\beta+1}{\beta}$ almost surely.
\end{proof}

\section{Convergence of loop erasures: proof of Theorem \ref{thm:Rayleigh conv}}\label{sec: lerw to rayleigh}

\subsection{Heuristics}

Recall the definition of a {\bf loop erasure} as given in \cref{sctn:LERW} and of \cref{sctn:gen Rayleigh def}. Let $k\in \NN$, let $n\in \NN$ be a large integer, and let $X$ be a maximal $k$-choice random walk on $K_n$ (recall that we  include self-loops) run from time $0$ to time infinity (recall the definition in Section \ref{sctn: crw}), with $\Ter = \emptyset$, and $\Av_n = \LE(X[0,n])$.

For every time $m\geq 0$, set $Z_m = \LE(X[0,m])$. Obviously, for every $m$, we have that the length of $Z_m$ is at most $n$. The aim of this section is to understand if the length of $Z_m$ stabilizes as $m$ grows to infinity, and if so, what is its typical length. To answer this question, we first assume that $Z_m$ has length $\ell$, and estimate the probability that the path will intersect itself in the next step. This intersection will occur if and only if all $k$ of the options sampled for the next step of the choice random walk are in $Z_m$. That is,
	\begin{equation*}
		\prcond{|Z_{m+1}| \leq \ell}{|Z_m| = \ell}{} = \left(\frac{\ell}{n}\right)^k,
	\end{equation*}
	and thus 
	\begin{equation*}
		\prcond{|Z_{m+1}| = \ell+1}{|Z_m| = \ell}{} = 1-\left(\frac{\ell}{n}\right)^k.
	\end{equation*}
Hence, by independence, the probability that $Z_m$ will grow from length $\ell$ to length, say, $2\ell$ in time $Z_{m+\ell}$ is (provided $\ell \ll n$)
	\begin{equation*}
		\prod_{p=\ell}^{2\ell-1}\left(1-\left(\frac{p}{n}\right)^k\right) \approx \left(1-\left(\frac{\ell}{n}\right)^k\right)^\ell \approx \exp\left(-\frac{\ell^{k+1}}{n^k}\right).
	\end{equation*}
	This probability will be roughly constant when $\ell^{k+1}$ is of order $n^k$, or in other words, when $\ell = n^{k/(k+1)}$. This is indeed the timescale appearing in Theorem \ref{thm:Rayleigh conv}.
 
\subsection{Finite dimensional convergence}
For ease of notation we now set $Z_m = |\LE(X[0,m])|$ (i.e. just the number of vertices in the loop-erasure, not the loop-erasure itself). In this subsection we will show that the process $Z_m$ converges after rescaling to the $k$-Rayleigh process in the Skorohod $J_1$-topology (introduced in Sections \ref{sctn:gen Rayleigh def} and \ref{sctn:Skor J1 def} respectively).
	We will follow the proof structure of an analogous result for a standard loop-erased random walk of \cite{Schweinsberg} to obtain this convergence. 
	
 To this end, fix $n \geq 1$, let $\Pi_k$ be a Poisson Point Process on $\RR^+ \times\RR^+$ with intensity $ky^{k-1} dx dy$ and let $(R^k_t)_{t \geq 0}$ denote the associated $k$-Rayleigh process as defined in Section \ref{sctn:gen Rayleigh def} (see Figure \ref{fig:Rayleigh choice}). For ease of notation we will use $R^k_t$ and $R^k(t)$ interchangeably. For every $1 \leq i\leq j$ let $I_{i,j}$ be the indicator that the box
	\begin{equation*}
		\left(\frac{j-1}{n^{k/(k+1)}},\frac{j}{n^{k/(k+1)}}\right) \times 	\left(\frac{i-1}{n^{k/(k+1)}},\frac{i}{n^{k/(k+1)}}\right)
	\end{equation*}
	is not empty in $\Pi_k$. Note that we have that the expected number of points in this box is 
	\begin{equation*}
		\frac{1}{n^{k/(k+1)}} \frac{i^k - (i-1)^k}{n^{k^2/(k+1)}} = \frac{i^k - (i-1)^k}{n^k},
	\end{equation*} 
	so that
	\begin{equation*}
		\pr{I_{i,j} \neq 0} = 1 - \exp\left(-  \frac{i^k - (i-1)^k}{n^k}\right).
	\end{equation*}
	It is useful at this stage to note that the expected number of points in the box corresponding to $I_{i,j}$ is equal to the probability that the choice walk on $K_n$ hits the $i^{th}$ vertex (in the loop-erasure) at some fixed step given that it has at least $i$ vertices in its loop-erasure up until that step. Indeed, suppose that the path has at least $i$ vertices, then the probability it hits one of the first $i$ vertices in the next step is $(i/n)^k$ (which is the probability that all choices for the next step are in the first $i$ vertices of this path). Similarly, the probability it will one of the first $(i-1)$ vertices is $\left((i-1)/n\right)^k$, and the difference between these quantities is exactly the probability it hits the $i^{th}$ vertex.

 The above observation will enable us to couple the evolution of the loop-erased choice random walk with the evolution of the $k$-Rayleigh process using the same underlying Poisson process $\Pi_k$. We now make this precise.
	

In particular, we now use $\Pi_k$ and more precisely the random variables $\left\{I_{i,j}\right\}_{i,j}$ to define a discrete process $(C_m)_{m \geq 1}$ that should
\begin{enumerate}[(1)]
    \item Be very close to the process $(R^k(t))_{t \geq 0}$,
    \item Have a law very similar to that of $(Z_m)_{m \geq 1}$.
\end{enumerate}

The process is defined as follows. First write $C_0 = 1$, then, given $C_{m-1}$, we define $M_m = \inf\{i \leq C_{m-1} \mid I_{i,m} \neq 0\}$ if the set is non-empty and $\infty$ otherwise. Then set
	\begin{equation*}
C_{m} = \begin{cases}
			C_{m-1} + 1 & \text{if}\quad M_m = \infty,
			\\ M_m & \text{if}\quad M_m < \infty.
		\end{cases}
	\end{equation*}
Fix $k \in \NN, \ell \in \NN$ and times $0 \leq t_1 < \ldots < t_\ell < \infty$ and write $t'_j = \lfloor t_j \cdot \nkk \rfloor$.
To show finite dimensional convergence, our first goal is to prove that $\left(Z_{t'_1},\ldots, Z_{t'_\ell}\right)/\nkk$ converges in distribution to $\left(R^k_{t_1},\ldots,R^k_{t_\ell}\right)$. We start with step (1) above.

We will mention the following lemmas from \cite{Schweinsberg}.

\begin{lemma}(cf. \cite{Schweinsberg} Lemma 2.7, Lemma 2.8).\label{lem:Rayleigh step 1}
For all integers $k \geq 1$ and $m \geq 0$, we have that 
\begin{equation}\label{eq: rayleigh c diff}
\left|\frac{C_m}{\nkk} - R^k\left({\frac{m}{\nkk}}\right)\right| \leq \frac{1}{\nkk}.
\end{equation}
and furthermore
\begin{equation*}
	\pr{\left|R^k_{t_i} - C_{t'_i}/\nkk\right| \geq \frac{3}{\nkk}}{} \leq {3(t_i+2)^k}n^{-\frac{k}{k+1}}.
\end{equation*}
\end{lemma}
\begin{proof}
The proof of \eqref{eq: rayleigh c diff} is deterministic and is exactly the same as in \cite[Lemma 2.7]{Schweinsberg}, only replacing the size of the square with $1/\nkk$ (note that $d_n = n^{-1/2}$ there).
To see why the second item does not hold immediately, note that we also need to control $R^k({t_i}) - R^k({t_i'/n^{\frac{k}{k+1}}})$. In particular, 
\begin{equation*}
t_i - t_i'/n^{\frac{k}{k+1}} \leq 2n^{-\frac{k}{k+1}},
\end{equation*}
and hence the only was that we can have $|R^k({t_i}) - R^k({t_i'/n^{\frac{k}{k+1}}})| \geq 2n^{-\frac{k}{k+1}}$ is if there are points in the Poisson process occurring below the process precisely in this narrow interval of time. Since $R^k({t}) \leq t$ for all $t \geq 0$ deterministically, it is enough to bound the event that there exist $j \in \{t'_i-1, t'_i, t'_i +1\}$ and $\ell \leq t'_i + 1$ such that $I'_{\ell,j} \neq 0$, which is upper bounded by $1-\exp\left(-{\frac{3(t_i+2)^k}{\nkk}}\right) \leq {3(t_i+2)^k}n^{-\frac{k}{k+1}}$.
\end{proof}

We now turn to step (2). We start with an elementary lemma.


\begin{lemma}\cite[Lemma 4.5]{schweinsberg2009loop} and \cite[Lemma 2.4]{Schweinsberg} \label{lem: copuling rv}
Let $Y_1, \ldots ,Y_j$, and $Q_1, \ldots, Q_j$ be $\{0,1\}$-valued random variables such that
\[
\pr{Y_i = 1 \text{ and } Y_{\ell}=0 \forall \ell \neq i}{} = p_i, \ \ \pr{Y_{\ell}=0 \forall \ell}{} = 1 - \sum_{i=1}^j p_i.
\]
Suppose that $(Q_i)_{i=1}^j$ are independent of each other and of $(Y_i)_{i=1}^j$ and that $\pr{Q_i = 1}{}=q_i$. Then exists a coupling of $\{Y_i\}$ and $\{Q_i\}$ such that
\begin{equation*}
	\pr{\exists i: Y_i \neq Q_i}{} \leq \sum_{i=1}^j |p_i-q_i| + \sum_{i \neq \ell}q_i q_{\ell}.
\end{equation*}
\end{lemma}

Using this lemma, we obtain the following.

\begin{lemma}(cf. \cite[Lemma 2.5]{Schweinsberg}).\label{lem:Rayleigh step 2}
	There exists a constant $c<\infty$ and a coupling of $\{C_i\}_{i=1}^{t'_{\ell}}$ and $\{Z_i\}_{i=1}^{t'_{\ell}}$ such that for all $n \geq 1$,
	\begin{equation*}
		\pr{\forall i\leq t'_{\ell} \quad Z_i = C_i}{} \geq 1 - cn^{-\frac{k}{k+1}}.
	\end{equation*}
\end{lemma}
\begin{proof}
This is again very similar to Lemma 2.5 in \cite{Schweinsberg}. We use \cref{lem: copuling rv} to bound the processes step by step and then take a union bound over all steps from times $1$ to $t_\ell$. We first recall that at time $T$, we have that both $Z_T$ and $C_T$ are at most $T$.

Let $p_i = \frac{i^k - (i-1)^k}{n^k}$ and $q_i = 1 -\exp\left({-p_i}\right)$. Note that $p_i$ is exactly the probability that the choice random walk will hit the $i$th vertex of its past loop-erasure at the next step, and $q_i$ is the probability that $I_{i,s} \neq 0$ (for any $s>0$).

Now suppose that we coupled $Z_m$ and $C_m$ so that they are the same up until some time $s$. In the notation of Lemma \ref{lem: copuling rv}, we set $Y_{i} = \mathbbm{1}\{Z_{s+1}=i\}$ and $Q_i = \mathbbm{1}\{ I_{i,s} \neq 0\}$. It therefore follows from Lemma \ref{lem: copuling rv} can then couple them until time $s+1$ with probability at least
\begin{align*}
1 - \sum_{i=1}^s |p_i-q_i| - 2\sum_{1 \leq i < \ell \leq s}q_i q_{\ell} \geq 1 - s(s-1)q_s^2 - \sum_{i=1}^{s}|p_i - q_i| \geq 1 - 2(s^2)p_s^2,
\end{align*}
hence when taking a union bound over all $s\leq t'_\ell$ we get that we can couple so that $Z_s = C_s$ for all $s\leq t'_\ell$ with probability at least
\begin{equation*}
	1 - (t'_\ell)^3 p_{t'_\ell}^2.
\end{equation*}
Finally, we can use the bounds $(t'_\ell)^3 \leq C_1^3 (\nkk)^3$ for some $C_1 > 0$ and 
\begin{equation*}
	p'_{t_\ell} \leq C_2\frac{(t'_\ell)^{k-1}}{n^k} \leq C_2 \left(\frac{C_1\nkk}{n}\right)^k\frac{1}{C_1\nkk} \leq C_2 \left(\frac{1}{\nkk}\right)^2,
\end{equation*}
to obtain
\begin{equation*}
	1 - (t'_\ell)^3 p_{t'_\ell}^2 \geq 1 - C\frac{1}{\nkk}.
\end{equation*}
\end{proof}

Combining Lemmas \ref{lem:Rayleigh step 1} and \ref{lem:Rayleigh step 2}, and letting $Z^{(n)}_t = \nkkk Z_{\lfloor t \nkk \rfloor}$ when we run the process on graph $K_n$, we obtain that there exists a probability space in which $\left(Z^{(n)}_{t_1},\ldots, Z^{(n)}_{t_\ell}\right)$ converges in probability to $\left(R^k_{t_1},\ldots,R^k_{t_\ell}\right)$ and hence also in distribution.

\subsection{Relative compactness}

To complete the proof, that is, to show convergence in the Skorohod-$J_1$ topology, we also have to show that the sequence of processes $(Z^{(n)})_{n \geq 1}$ is relatively compact. This can be achieved using Proposition \ref{prop:billingsley tightness criteria} (also check \cite[Section 3]{Schweinsberg}): part (a) is trivial since $\nkkk Z_{t\nkk} \leq t$ deterministically.

The non-trivial part is to show (b), that is for all $\eps>0$ and $T>0$ there exists $\theta = \theta(\eps,T) > 0$ such that
\begin{equation}\label{eq: rel comp cond}
	\limsup \pr{w(Z^{(n)},\theta,T) \geq \eps}{} \leq \eps,
\end{equation}
where $w(Z^{(n)},\theta,T) = \inf \max_i \sup_{t,u \in [t_{i-1},t_i)} |Z^{(n)}_t - Z^{(n)}_u|$, and the infimum is taken over all $m>0$ and all $0=t_0 < \ldots < t_m \leq T$ such that $t_i - t_{i-1} \geq \theta$ for all $i$.

We note that since the rescaled (in both space and time) process $Z^n$ grows at most linearly, the difficulty is bounding $Z^{(n)}_t - Z^{(n)}_u$ from above when $t < u$, as the process might make large jumps downwards. 
It thus suffices to bound the probability that we see two large jumps in a time interval of length $\theta$ such a jump. This is a very straightforward calculation.

In every single non-scaled step occurring at a time smaller than $T\nkk$ of the choice random walk, the probability to make a jump downwards in the next step is at most $\left(\frac{T\nkk}{n}\right)^k = T^k\nkkk$. Hence, the probability to make two such jumps in an interval of non rescaled length $2\theta \nkk$ is at most $4\theta^2 T^{2k}$. The probability there exist times have $t_1 < t_2 < t_1 + 2\theta\nkk < T\nkk$ such that the choice random walks jumps downwards at times $t_1$ and $t_2$ is thus bounded by the probability that there exists an interval of the form $[i\theta\nkk, (i+1)\theta\nkk]$ with two downward jumps. By a union bound, this probability is at most $4\theta T^{2k+1}$, and hence we can choose $\theta$ so that this probability is smaller than $\eps$, to obtain \eqref{eq: rel comp cond} for the rescaled processes.

\section{GHP convergence of Aldous-Broder choice trees: proof of \cref{thm:AB convergence}}\label{sec: AB trees}

In this section we prove Theorem \ref{thm:AB convergence}. This is broken up into convergence of the partial subtrees (finite dimensional convergence) and then a tightness condition.

\subsection{Finite-dimensional GH convergence}

Fix some $k \geq 2$. Recall the definition of the times $(\sigma_i(G))_{i \geq 0}$ and the partial spanning trees $(\TAB_i(G))_{i \geq 0}$ from \eqref{eqn:sigma i def} and \eqref{eqn:AB partial spanning tree} obtained from running the $k$-choice \AB algorithm on a graph $G$. Note that the trees are indexed so that $\TAB_i (G)$ has $i$ branches. For each $n \geq 1$, we let $\sigma_{i,n}$, $\TABin$ and $\TABn$ denote $\sigma_i(K_n)$, $\TAB_i(K_n)$ and $\TAB(K_n)$ where $K_n$ is the complete graph with $n$ vertices. We also let $d_n$ denote the graph metric on $\TABn$.

Note: for ease of notation, we will assume throughout this section that $k$ is fixed, as is the choice of whether we are using the maximal or uniform algorithm (and these will be the same at every step of the algorithm). Also, recall that $K_n$ has self loops.

We will additionally set $\beta=k$, $\gamma=k-1$, and for each $i \geq 1$ let $T^{(i)}$ denote the partial tree obtained after running $i$ steps of the stick breaking algorithm as in Definition \ref{def:stick breaking}.

The main aim of this subsection is to prove the following.
\begin{proposition}\label{prop:finite trees converge AB}
Fix some $i \geq 1$ and consider the sequence $(\TABin)_{n \geq 0}$. Then 
\[
n^{-\frac{k}{k+1}}\TABin \overset{(d)}{\to} T^{(i)}
\]
as $n \to \infty$ with respect to the Gromov--Hausdorff topology.
\end{proposition}

The proof rests on the following proposition, which will ultimately allow us to couple the first $i$ steps of the $k$-choice \AB algorithm with the first $i$ steps of the stick breaking process.

\begin{proposition}
\label{prop:main stick breaking ingredient}
Take any $\epsilon\in (0,\infty)$, $B\in (0,\infty)$ and $D \in (0, \infty)$.
\begin{enumerate}[(1)]
    \item Suppose that $C \in [0, D]$. Then for any $b \leq B$, $v \in K_n$,
    \begin{align*}
        \prcond{\sigma_{i+1,n} - \sigma_{i,n} \geq Cn^{\frac{k}{k+1}}}{|\TABin|=bn^{\frac{k}{k+1}}, X_{\sigma_{i,n}}=v}{} = \exp\left\{- \frac{(C + b)^{k+1}-b^{k+1}}{k+1} \right\} ( 1 + o(1)),
    \end{align*}
    where the $o(1)$ might depend on $B$ and $D$, but is uniform over all other choices.
     \item Let $H_{i,n} = \inf\{ m: X_{\sigma_{i,n}} = X_{m n^{\frac{k}{k+1}}}\}$. Then for any connected interval $A \subset (\epsilon, \infty)$, and all $b \in [0, \infty)$,
\begin{align*}
\prcond{H_{i,n} \in A}{|\TABin|=bn^{\frac{k}{k+1}}}{u} \to \begin{cases}
  \frac{k}{b^k} \int_{ A \cap [0,b]} a^{k-1} da \text{ for the maximal choice,} \\
     \frac{\textsf{Leb} (A \cap [0,b])}{b} \text{ for the uniform choice,} 
\end{cases}
\end{align*}
where the convergence is uniform over all choices of $A$ and $b$ for fixed on $\epsilon$.
\end{enumerate}
\end{proposition}
\begin{proof}
\begin{enumerate}[(1)]
    \item 
    Note that, by symmetry and independence of each of the $k$ choices, we have for each $b \geq 0$ that
    \begin{align*}
      \prcond{\sigma_{i+1,n} - \sigma_{i,n} \geq  m}{\sigma_{i+1,n} - \sigma_{i,n} \geq m-1}{} = 1- \left(\frac{|\TABin|+m-1}{n}\right)^k.
    \end{align*}
    Hence by multiplying we obtain that 
    \begin{align*}
    \prcond{\sigma_{i+1,n} - \sigma_{i,n} \geq Cn^{\frac{k}{k+1}}}{|\TABin|=bn^{\frac{k}{k+1}}, X_{\sigma_{i,n}}=v}{} &= \prod_{m=1}^{Cn^{\frac{k}{k+1}}} \prcond{\sigma_{i+1,n} - \sigma_{i,n} \geq m}{\sigma_{i+1,n} - \sigma_{i,n} \geq m-1}{} \\
    &= \prod_{m=1}^{Cn^{\frac{k}{k+1}}} \left(1-\left(\frac{bn^{\frac{k}{k+1}}+m-1}{n} \right)^k\right)\\
    &= \prod_{m=1}^{Cn^{\frac{k}{k+1}}} \exp\left\{-\left(\frac{bn^{\frac{k}{k+1}}+m-1}{n}\right)^k(1-o(1))\right\} \\
    &= \exp\left\{-\left(\frac{(1-o(1))}{n^k}\right) \left((bn^{\frac{k}{k+1}}+Cn^{\frac{k}{k+1}}-1)^{k+1} - (bn^{\frac{k}{k+1}})^{k+1}  \right)\right\} \\
    &= \exp\left\{- \frac{(C + b)^{k+1}-b^{k+1}}{k+1} \right\} ( 1 + o(1)).
    \end{align*}
\item We first claim that
\begin{align}\label{eqn:hitting prob discrete AB}
\prcond{H_{i,n} \in A}{|\TABin|=bn^{\frac{k}{k+1}}}{u} = \begin{cases}
     \sum_{a \in A \cap [0,b]: an^{\frac{k}{k+1}} \in \NN \cup \{0\}} \left(\frac{an^{{\frac{k}{k+1}}}+1}{bn^{{\frac{k}{k+1}}}+1}\right)^k - \left(\frac{an^{\frac{k}{k+1}}}{bn^{\frac{k}{k+1}}+1}\right)^k \text{ for the maximal choice,} \\
     \frac{|a \in A \cap [0,b]: an^{\frac{k}{k+1}} \in \NN \cup \{0\}|(1+o(1))}{bn^{\frac{k}{k+1}}} \text{ for the uniform choice,} 
\end{cases}
\end{align}
where the $o(1)$ is uniform over all $b \leq B$ and $A \subset (\epsilon, \infty)$. The formula can be explained as follows: for the uniform choice, we know that $X_{\sigma_{i,n}}$ is chosen uniformly from the neighbours of $X_{\sigma_{i,n}-1}$ that are also in $\{X_m: 0 \leq m < \sigma_{i,n}-1\}$. The size of the set $\{X_m: 0 \leq m < \sigma_{i,n}-1\}$ is $|\TABin|$. The probability that $X_{\sigma_{i,n}}$ is equal to a particular vertex in this set is therefore equal to $\frac{1}{|\TABin|}$
which leads to the desired conclusion (since $i$ is fixed).

In the case of maximal choice, $X_{\sigma_{i,n}}$ is chosen from the same set but has a different law. In particular, identifying vertices with the time in which they first appeared in the random walk trajectory, choosing $X_{\sigma_{i,n}}$ is equivalent to choosing a time in the set
\[
S=\{0, 1, \ldots, \sigma_{i,n}-1\} \setminus \{\sigma_{j,n}: 1 \leq j \leq i-1\}.
\]
Note that $|S|=|\TABin|$. For every $0\leq j \leq \sigma_{i,n}-1$, let $N(j) = |\{\ell \in S: \ell < j\}|$. Note that, necessarily,
\begin{equation}\label{eqn:Nj props}
j - (i-1) \leq N(j) \leq j \quad \text{and} \quad N(j+1)-N(j) = 1 \text{ for all but at most } i \text{ values of } j.
\end{equation}
for all $j$. Then note that the probability that we chose a point in $S$ that is strictly smaller than $j$ is equal to the probability that all of the $k$ choices for the next random walk step were strictly less than $j$. Hence if $j = an^{\frac{k}{k+1}}$ this probability is given by $\left( \frac{N(j)}{|\TABin|}\right)^k$. For all but at most $i$ values of $j$, the probability that we choose exactly $j$ is therefore given by
\[
\left( \frac{N(j+1)}{|\TABin|}\right)^k - \left( \frac{N(j)}{|\TABin|}\right)^k = \left( \frac{N(j)+1}{|\TABin|}\right)^k - \left( \frac{N(j)}{|\TABin|}\right)^k.
\]
Combining with \eqref{eqn:Nj props} establishes the claim in \eqref{eqn:hitting prob discrete AB}.

To take the limits, we treat the uniform case first as it is simplest. Note that it follows firstly if $A$ is an interval that 
\[
n^{-\frac{k}{k+1}}|a \in A \cap [0,b]: an^{\frac{k}{k+1}} \in \NN \cup \{0\}| \to \textsf{Leb} (A \cap [0,b]),
\]
from which it follows that the probability in question converges to 
\[
\frac{\textsf{Leb} (A \cap [0,b])}{b}.
\]
For the maximal choice, note that for any $a \in A$ we can write:
\[
 \left(\frac{an^{{\frac{k}{k+1}}}+1}{bn^{{\frac{k}{k+1}}}+1}\right)^k - \left(\frac{an^{\frac{k}{k+1}}}{bn^{\frac{k}{k+1}}+1}\right)^k = \frac{1}{(bn^{\frac{k}{k+1}})^k} k(an^{\frac{k}{k+1}})^{k-1} [1+ o(1)] =  \frac{ka^{k-1} [1+ o(1)]}{b^k n^{\frac{k}{k+1}}},
\]
where the $o(1)$ is uniform over all $b \leq B$. Then note that
\begin{align*}
\frac{k (1+ o(1))}{b^k n^{\frac{k}{k+1}}} \sum_{a \in A \cap [0,b]: an^{\frac{k}{k+1}} \in \NN \cup \{0\}} a^{k-1} = \frac{k (1+ o(1))}{b^k} \int_{ A \cap [0,b]} a^{k-1} da.
\end{align*}
\end{enumerate}
\end{proof}

In Definition \ref{def:stick breaking} we defined how a sequence of trees can be constructed through a stick-breaking process. In what follows next we outline how, for any $i \geq 1$, the choice \AB algorithm on $K_n$ can be used to give two sequences $(Y_j)_{j = 0}^n$ and $(Z_j)_{j = 0}^{n-1}$ such that $\SB^{(i)} ((Y_0, Y_1, \ldots, Y_{i}), (Z_0, Z_1, \ldots, Z_{i-1} ))$ has exactly the same tree structure as the subtree obtained after the first $i$ steps of the algorithm (i.e. it has exactly the same branch lengths and branchpoint locations).

    We begin with $Y_0=0, Z_0=0$, and run the choice \AB algorithm as described in Section \ref{sctn:AB alg def}. Recall also from \eqref{eqn:sigma i def} the definition of the times $(\sigma_i)_{i \geq 0}$, and that we defined in \eqref{eqn:AB partial spanning tree}
\begin{equation*}
\TAB_i (K_n) = \{(X_m, X_{m+1}): 1 \leq m < \sigma_{i,n} \text{ and } \nexists k \leq m \text{ such that } X_k = X_{m+1}\}.
\end{equation*}
For $i \geq 1$, we can more precisely define the $i^{th}$ branch as either an edge set or a set of ``new" vertices via
\begin{align*}
\branch_i &= \{(X_m, X_{m+1}): \sigma_{i-1,n} \leq m < \sigma_{i,n} - 1\}, \qquad \branch^{\text{new}}_i = \{X_{m}: \sigma_{i-1,n} < m < \sigma_i \},
\end{align*}
Note that it is then the case that $\TAB_i = \TAB_{i-1} \cup \branch_i$. For any vertex $v \in \branch_i$, we say that $v$ was added at the $i^{th}$ step. Moreover, we set 
\begin{align}\label{eqn:Yi Zi def}
Y_i = n^{-\frac{k}{k+1}}|\TABin|, \qquad Z_i = n^{-\frac{k}{k+1}}\inf \{m: X_m = X_{\sigma_{i,n}}\}
\end{align}
It then follows by construction that $\SB^{(i)} ((Y_0, Y_1, \ldots, Y_{i}), (Z_0, Z_1, \ldots, Z_{i-1} ))$ is equal in law to $\TABin$ (in terms of its metric space structure).

We note that some notational complications arise in the case where $Y_i=Y_{i+1}$, but since this does not happen with high probability we will ignore this problem.

It will more generally be useful to define, for $v \in K_n$,
\[
I(v) = \inf \{m: X_m = v\}
\]
to denote the time step at which $v$ was added to the tree.

\begin{proposition}\label{prop:Prokh conv measure on sticks}
Fix $i \geq 1$ and let $I_{\max} = \max_{v \in \TABin}I(v) = Y_i - 1$.  Take any $j \geq 1$ and let $\Pb_{d,j}$ and $\Pb_{k}$ be two measures on $[0, I_{\max}]$ defined by
	\begin{align*}
	   \prstart{m}{d,j} &= \prcond{Z_i = m}{Y_i=j}{} \ \ \ \forall m \leq I_{\max}, \\
	   \prstart{A}{k} &= \int_A \frac{ku^{k-1}}{I_{\max}^{k}} \ \ \ \forall A \subset [0, I_{\max} ],
	\end{align*}
 and let $\Pb_{u}$ denote uniform measure on $[0, I_{\max}]$.
	Then, for every $\eps>0$ there exists $N\in \NN$ such that for all $n>N$ and for all $j $, the Prokhorov distance between the measure $\Pb_{d,j}$ and $\Pb_k$ is less than $\eps$ for the maximal choice, and the Prokhorov distance between the measure $\Pb_{d,j}$ and $\Pb_u$ is less than $\eps$ for the uniform choice.
	\end{proposition}
	\begin{proof}
Fix $\eps>0$. First note that we can assume that $I_{\max} \leq \sqrt{\eps}$; otherwise the Prokhorov distance is automatically upper bounded by $\sqrt{\eps}$ (which proves the claim since $\eps$ was arbitrary). The proof is now just a straightforward rephrasing of \cref{prop:main stick breaking ingredient}(2); we write it for the maximal choice. In particular, $\{\sigma_{i,n}=bn^{\frac{k}{k+1}}\} \iff \{Y_i=b\}$, and hence it follows that for any $A \subset [\eps, I_{\max}]$, 
 \[
 \Pb_{d,j}(A) \to \Pb_k(A)
 \]
uniformly over all choices of $A$. We can also include the set $A = [0, \eps]$ in this by considering its complement. Hence we have for any $A \subset [\eps, I_{\max}]$ that, provided $n$ is sufficiently large,
 \[
 \Pb_{d,j}(A) \leq \Pb_k(A) + \eps + \Pb_{d,j}([0,\eps]) \leq \Pb_k(A) + \eps + \frac{\eps^k}{I_{\max}^k} \leq \Pb_k(A) + \eps + \eps^{k/2},
 \]
 and
 \[
 \Pb_{d,j}(A) \geq \Pb_k(A) - \eps - \Pb_{k}([0,\eps]) \geq \Pb_k(A) - \eps - \eps^{k/2},
 \] 
 which establishes the claim (since $\eps$ was arbitrary). The proof is the same in the uniform case.
	\end{proof}

In the next claim, we keep $\beta$ and $\gamma$ as in Proposition \ref{prop:finite trees converge AB}, and let $(Y_0', Y_1', Y_2', \ldots)$ and $(Z_0', Z_1', Z_2', \ldots )$ be sampled as described in Definition \ref{def:random stick breaking}. We also let $(Y_i)_{i \geq 0}$ and $(Z_i)_{i \geq 1}$ be as defined in \eqref{eqn:Yi Zi def}.

\begin{proposition}
    \label{claim:coupling of stick breaking}
For every $\epsilon>0$ and $i \geq 1$ there exists $N$ such that for all $n>N$ we can couple the stick-breaking process and the spanning tree construction such that $|Y_j - Y_j'| \leq \epsilon$ for all $0 \leq j\leq i$ and $|Z_j- Z_j'| \leq \epsilon$ for every $0 \leq j\leq i-1$ with probability at least $1-\epsilon$.
\end{proposition}
\begin{proof}
We prove the claim by induction. When $i=1$ (i.e. for $T^{(1)}_n$) the tree is a single branch and by Proposition \ref{prop:main stick breaking ingredient} the probability that its rescaled length exceeds $x$ converges to $e^{-x^{k+1}/(k+1)}$, which is also equal to the probability that $Y_1'$ exceeds $x$. The result in this case therefore follows from Lemma \ref{lem: proh close rv}.

Now fix $i\geq 2$ and suppose that the claim holds for all $j<i$. We will now show that the claim holds also for $i$. We can therefore fix some $\zeta< \frac{\eps}{8}$ (the precise value of $\zeta$ will be chosen later) and suppose that there exists $N$ such that for all $n>N$ we can successfully couple the two processes so that, with probability at least $1-\zeta$:
\begin{align*}
|Y_j - Y_j'| \leq \zeta \text{ for all } 0 \leq j\leq i-1 \text{ and } |Z_j- Z_j'| \leq \zeta \text{ for every } 0 \leq j\leq i-2.
\end{align*}
Note that it follows from Definition \ref{def:random stick breaking} that there exist functions $f, g$ where $g(\eps)>0$ and $f(\epsilon)<\infty$ and such that $0<g(\eps) \leq Y_{i-1} \leq  f(\eps)$ with probability at least $1-\epsilon/8$. Hence we can also assume that this is the case.

Take any (large) $D < \infty$. Under the coupling, we can write  $\frac{|\TABiin|}{n^{\frac{k}{k+1}}} = Y_{i-1} = Y'_{i-1} + \eps'$ where $\epsilon' \in [-\zeta, \zeta]$. Therefore we can apply Proposition \ref{prop:main stick breaking ingredient}(1) with $B= f(\epsilon)$, to get that for every $C \in [0, D]$,
\begin{align}\label{eq: approximating Y}
        \prcond{Y_{i} - Y_{i-1} > C }{\TABiin}{} &= \exp\left\{- \frac{(C + Y'_{i-1} + \eps')^2-(Y'_{i-1} + \eps')^2}{2} \right\} + o(1) \nonumber\\
        &=\exp\left\{- \frac{(C + Y'_{i-1})^2-(Y'_{i-1})^2}{2} \right\} \left(1+e^{-C\epsilon'} - 1 \right) + o(1),
        \end{align}
so
\begin{align}\label{eq: distance between Ys}
\left| \prcond{Y_{i} - Y_{i-1} > C }{\TABiin}{} - \exp\left\{- \frac{(C + Y'_{i-1})^2-(Y'_{i-1})^2}{2} \right\} \right| \leq |1-e^{-C\epsilon'}|e^{\frac{-C^2}{2}} + o(1) \leq C\zeta e^{\frac{-C^2}{2}} + o(1).
     \end{align}
The term $C\zeta e^{\frac{-C^2}{2}}$ on the right hand side of \eqref{eq: distance between Ys} goes to $0$ as $\zeta\to 0$ uniformly over $C\in [0,D]$. We would like to apply Lemma \ref{lem: proh close rv}, which implies that there exists $\eta$ depending on $\eps$ such that if $g(\eps) \leq Y_{i-1} \leq f(\eps)$ and the left-hand side of \eqref{eq: distance between Ys} (see \eqref{eqn:variables close}) is smaller than $\eta$ for all $C\in (0,\infty)$, then we can couple $Y_{i} - Y_{i-1}$ and $Y'_{i} - Y'_{i-1}$ such that the probability that they are $\eps/8$ close to one another is at least $1-\eps/8$. When this happens, by the triangle inequality, we have that $|Y_{i} - Y'_{i}| < \eps / 2$. In order to do this, note that when we choose $\zeta$ small enough (and smaller than $\eps / 8$) and $n$ large enough such that the right-hand side of \eqref{eq: distance between Ys} is smaller than this $\eta$, we have \eqref{eqn:variables close} for all $C \leq D$. If we further increase $D$ (and then $n$ if necessary), we get for all possible trees and for all $C\geq D$ that the probabilities of $\{Y_{i} - Y_{i-1} > C\}$ and of $\{Y'_{i} - Y'_{i-1} > C\}$ are smaller than $\eta/2$, giving \eqref{eqn:variables close} for all $C\in(0,\infty)$, so we can indeed apply Lemma \ref{lem: proh close rv} as described to get the desired coupling.

However, we note that $Z_{i-1}$ is not independent of $Y_{i}$ and we are required to couple the pair $(Y_{i}, Z_{i-1})$ with $(Y'_{i}, Z'_{i-1})$.
To do so, we will decompose $\RR^+$ into intervals of length $\eps/2$, that is, we write $\RR^+ = \bigcup_{j=0}^{\infty} I_j$ where $I_j = [j\eps/2, (j+1)\eps/2)$. Let $M_{i}$ (respectively $M'_{i}$) be the unique $j$ such that $Y_{i} \in I_j$ (respectively $Y'_{i} \in I_j$).
By Lemma \ref{lem: proh close rv} and the discussion above, there exists a coupling of $M_{i}$ and $M'_{i}$ such that the difference between them is at most $1$ with probability $1-\eps/2$. Moreover, by Definition \ref{def:random stick breaking}, with probability at least $1-\eps/8$ we can increase $B$ if necessary so that $M'_{i} \leq B$ with probability at least $1-\eps/8$ (and then $M_{i} \leq B+1$). 

Then, given $M_{i}$, we sample $Z_{i-1}$ according to its conditional law. By Proposition \ref{prop:Prokh conv measure on sticks}, when $n$ is large enough, for every $j\leq B+1$, conditionally on $M_{i}=j$ we have that the Prokhorov distance between $Z_{i-1}$ and a uniform random variable on $[0,Y_{i-1}]$ is at most $\zeta$ in the uniform case, and similarly with $\Pb_k$ in the maximal case. Therefore, the Prokhorov distance between $Z'_{i-1}$ and $Z_{i-1}$ conditionally on $M_{i} = j$ is at most $2\zeta$. Since $\zeta < \eps/8$, and by summing the errors of $\epsilon/8$ accrued at each of the previous steps, it follows that we can couple the pairs $(Y_{i}, Z_{i-1})$ and $(Y'_{i}, Z'_{i-1})$ such that $|Y_{i} - Y'_{i}| < \eps$ and $|Z'_{i-1} - Z_{i-1}| < \eps$ with probability at least $1 - \eps$, as required.
\end{proof}

\begin{corollary}
For every $\epsilon>0$ and $i \geq 1$ there exists $N$ such that for all $n>N$ the following holds. Let $(Y_{\ell})_{{\ell}=1}^{i}, (Y'_{\ell})_{{\ell}=1}^{i}, (Z_{\ell})_{{\ell}=1}^{i}$ and $(Z_{\ell}')_{{\ell}=1}^{i-1}$ be as above, let $d$ denote distance on $\SB^{(i)} ((Y_0, \ldots, Y_i), (Z_0, \ldots, Z_{i-1} ))$ and similarly let $d'$ denote distance on $\SB^{(i)} ((Y_0', \ldots, Y_i'), (Z_0', \ldots, Z_{i-1}' ))$, using the definitions of Definition \ref{def:stick breaking}. Then we can couple these stick-breaking processes such that, with probability at least $1-\epsilon$, it holds for all $0 \leq {\ell},j \leq i$ that 
\[
|d(Y_{\ell}, Y_j) - d'(Y_{\ell}', Y_j')| \leq \epsilon.
\]
\end{corollary}
\begin{proof}
Take $\eta>0$ and $i \geq 1$. We verify that there is a coupling such that each of the conditions of \cref{prop:stick breaking close} hold with high probability.

For the first condition note that, by Proposition \ref{claim:coupling of stick breaking}, we can couple the stick-breaking process for the CRT and for the $\UST$ such that $|Y_{\ell} - Y_{\ell}'| \leq \eta$ for all $0 \leq {\ell}\leq i$ and $|Z_{\ell}- Z_{\ell}'| \leq \eta$ for every $0 \leq {\ell}\leq i-1$ with probability at least $1-\eta$ for all sufficiently large $n$. For the second condition, note that it follows from Definition \ref{def:random stick breaking} that we can choose $\delta=\delta(\eta, i)>0$ such that $|Z_{\ell}' - Y_j'| \geq 3\eta$ for all ${\ell} \leq i-1,j \leq i$ with probability at least $1-\delta$, and such that $\delta \downarrow 0$ as $\eta \downarrow 0$.

Therefore, it follows from Proposition \ref{prop:stick breaking close} that under this coupling, it holds with probability at least $1-\eta-\delta$ that $\sup_{1 \leq {\ell},j \leq i}|d(Y_{\ell}, Y_j) - d'(Y_{\ell}', Y_j')| \leq 2i\eta$. Given $\epsilon>0$, we can therefore choose $\eta>0$ small enough that $2i\eta+\delta<\epsilon$ in order to deduce the claim as stated.
\end{proof}

\subsection{Tightness for the GH topology}

The proof of tightness follows that used by Aldous in \cite[Section 4]{AldousCRTI} to prove similar results for the stick-breaking construction of the CRT. This is also the strategy we followed in Section \ref{sctn:stick breaking defs}.

For this we start by defining the projection $\rho_n^{\text{AB}} : [0, \infty) \to \TABn $, by $\rho_n^{\text{AB}} (t) = \left\{X_{\lfloor tn^{\frac{k}{k+1}}\rfloor}\right\} $. We also define the notion of the subtree formed at time $t$, generalizing slightly the definition in \eqref{eqn:AB partial spanning tree}. In particular, for each $t \geq 0$ we set 
\begin{equation*}\label{eqn:AB partial spanning tree time t}
\TABn (t) = \{(X_m, X_{m+1}): 1 \leq m < t \text{ and } \nexists k \leq m \text{ such that } X_k = X_{m+1}\}
\end{equation*}
(so that $\TABin = \TABn (\sigma_{i,n})$).

Recall also that $d_n$ denotes the graph metric on the tree $\TABn$. For the next few propositions, we define $D_n$ by 
\[
D_n(s,t) = \sup_{0 \leq r \leq s} n^{-\frac{k}{k+1}} d_n(\rho_n^{\text{AB}}(s), \rho_n^{\text{AB}}(t)) \qquad 0<s<t.
\]

\begin{proposition}\label{prop:tightness AB exp tail}(cf \cite[Lemma 5 and Lemma 9]{AldousCRTI}).
For all $s\geq 0, n \geq 1$ there is an event $A_n(s)$ satisfying $\pr{A_n(s)} \geq 1-  e^{-sn^{\frac{1}{k+2}}}$ and such that for all $t>s$ and all $x>0$,
\[
\prcond{D_n (s,t) \geq x}{A_n(s)}{} \leq e^{-\frac{x}{3}s^k}.
\]
\end{proposition}
\begin{proof}
For $n \geq 1$ and $s \geq 0$ set
\[
\sigma_n (s) = \sup\{i \geq 1: \sigma_{i,n} \leq sn^{\frac{k}{k+1}}\}, \qquad A_n(s)=\{\sigma_n(s) \geq sn^{\frac{1}{k+1}}\}.
\]
Note that $\sigma_n (s)$ is stochastically dominated by a \textsf{Binomial}($\lfloor sn^{\frac{k}{k+1}} \rfloor, n^{\frac{-k}{k+1}} $) random variable, and hence, by a Chernoff bound,
\[
\pr{A_n(s)} \leq \exp \{ s(e-1) - sn^{\frac{1}{k+1}}\}.
\]
We claim that on the event $A_n(s)$, the quantity $n^{\frac{k}{k+1}}D_n (s,t)$ is dominated by a \textsf{Geometric}($\frac{1}{2}sn^{\frac{k}{k+1}}$) random variable. It is sufficient to fix $s$ and prove this by induction for all $t$ of the form $in^{-\frac{k}{k+1}}$ where $i \geq 1$. To do this, we again use the strategy of \cite[Lemma 9]{AldousCRTI} that we already used in Proposition \ref{prop:tightness stick breaking exp tail}. We again use a partition of the form \eqref{eqn:domination Aldous geometric} but this time with $\eps=1$ (this is exactly \cite[Equation (37)]{AldousCRTI} in fact). Now $A$ is the event that the next step of the choice random walk on $K_n$ is in the set $\{X_r: 0 \leq r \leq \lfloor sn^{\frac{k}{k+1}} \rfloor \}$, $B$ is the event that the next step of the random walk is a new vertex, and $C$ is the event that it is in the set $\{X_r: \lfloor sn^{\frac{k}{k+1}} \rfloor < r \leq \lfloor tn^{\frac{k}{k+1}} \rfloor \}$.

In the case $i=1$ the quantity is deterministically at most $1$ so the claim is trivially satisfied; to complete the induction we therefore just need to observe that
\[
\pr{A} = \left(\frac{\lfloor sn^{\frac{k}{k+1}} \rfloor - \sigma_n(s)}{n-1}\right)^k \geq \frac{1}{2}s^kn^{-\frac{k}{k+1}}
\]
by the definition of the choice random walk and working on the event $A_n(s)$. Note that the required independence assumptions to apply the induction hold by the Markov property. In the notation of \eqref{eqn:domination Aldous geometric} we set $r=in^{-\frac{k}{k+1}}$, and on the event $C$ we take $\eta_{r-1}$ to be a weighted average over the laws of $\psi_{r'}$ for $r' \in [s,r-1]$; these all satisfy the required domination by the inductive hypothesis (as does $\psi_{r-1}$ for the event $B$). 

Hence we deduce that 
\[
\prcond{D_n (s,t) \geq x}{A_n(s)}{} \leq \left(1-\frac{1}{2}s^kn^{-\frac{k}{k+1}}\right)^{\lfloor xn^{\frac{k}{k+1}}\rfloor } \leq e^{-\frac{x}{3}s^k},
\]
as claimed.
\end{proof}

\begin{corollary}(cf \cite[Lemma 6]{AldousCRTI})\label{cor:Aldous sup D interval}
    \[
    \pr{\sup_{e^j \leq t \leq e^{j+1}} D_n (e^j, t) > 3(k+2)je^{-jk}} \leq e^{k+3-j}.
    \]
\end{corollary}
\begin{proof}
Set $s=e^j, u=e^{j+1}$ and
\[
\tau_n = \tau_n(x) = \inf\{t \geq s: D(s,t) \geq x\}.
\]
On the event $A_n(s)$ we have by Proposition \ref{prop:tightness AB exp tail} that, for any $s<u$ and any $x>0$:
    \begin{align}\label{eqn:Aldous exp bound interval UB}
        \E{\mathbbm{1}\{A_n(s)\} \textsf{Leb} \{s \leq t \leq u+1: D(s,t) \geq x\}} \leq (u+1-s)e^{-\frac{x}{3}s^k}.
    \end{align}
However, by the same logic as in the proof of Proposition \ref{prop:tightness AB exp tail}, it also follows that 
\begin{align*}
\{\lfloor t n^{\frac{k}{k+1}} \rfloor \in [n^{\frac{k}{k+1}}\tau_n, \sigma_{\sigma_n(\tau_n)+1,n}]\} \subset \{D(s,t) \geq x\}.
\end{align*}
Hence
\begin{align}\label{eqn:Aldous exp bound interval LB}
\begin{split}
 &\E{\mathbbm{1}\{A_n(s)\} \textsf{Leb} \{s \leq t \leq u+1: D(s,t) \geq x\}} \\
 &\geq \E{\mathbbm{1}\{\tau_n \leq u+1\} (\sigma_{\sigma_n(\tau_n)+1,n} \wedge (u+1) - n^{\frac{k}{k+1}}\tau_n)} \\
 &\geq \pr{\tau_n \leq u+1}\econd{(\sigma_{\sigma_n(\tau_n)+1,n}\wedge (u+1)  - n^{\frac{k}{k+1}}\tau_n)}{\tau_n \leq u+1} - \pr{A_n(s)^c}.
 \end{split}
\end{align}
Now note that, for any $t$, 
\begin{align*}
    \E{(\sigma_{\sigma_n(t)+1,n}\wedge (u+1)  - n^{\frac{k}{k+1}}t)} &\geq \int_0^{u+1} \pr{ (\sigma_{\sigma_n(t)+1,n}\wedge (u+1)  - n^{\frac{k}{k+1}}t) >s} ds \\
    &\geq \int_0^{u+1} \pr{\textsf{Poi}(ru^k)=0} dr \\
    &\geq  u^{-k}(1-e^{-(u+1)u^k}).
\end{align*}
By the strong Markov property, we can substitute this into \eqref{eqn:Aldous exp bound interval LB} and then compare to \eqref{eqn:Aldous exp bound interval UB} to deduce that
\begin{align*}
    \pr{\tau_n \leq u+1} \leq \frac{(u+1-s)e^{-\frac{x}{3}s^k} + \pr{A_n(s)^c}}{u^{-k}(1-e^{-(u+1)u^k})} \leq e^{k+2} s^{k+1} (e^{-\frac{x}{3}s^k}+e^{-n^{\frac{k}{k+1}}s}).
\end{align*}
In particular, reintroducing the dependence on $x$ and taking $x=\frac{3(k+2)\log s}{s^k}$ gives
\begin{align*}
    \pr{\tau_n\left(\frac{3(k+2)\log s}{s^k}\right) \leq u+1} \leq e^{k+3} s^{-1} = e^{k+3-j}.
\end{align*}
\end{proof}

This allows us to deduce the following.

\begin{corollary}\label{cor:AB tightness}
For all $\epsilon>0$,
\[
\lim_{i \to \infty} \limsup_{n \geq 1} \pr{n^{-\frac{k}{k+1}
} d_{GH}(\TABin, \TABn) \geq \eps} = 0.
\]
\end{corollary}
\begin{proof}
By summing over $j \geq J$ and applying a union bound we deduce from Corollary \ref{cor:Aldous sup D interval} that
 \[
    \pr{\sup_{e^J \leq t} D_n (e^J, t) > 3(k+2)e^ke^{-J(k-1)}} \leq \pr{\exists j \geq J: \sup_{e^j \leq t \leq e^{j+1}} D_n (e^j, t) > 3(k+2)je^{-jk}} \leq e^{k+3-J}.
    \]
In particular this implies that $ \pr{d_{GH}(\TABn(e^J), \TABn) \geq 3(k+2)e^ke^{-J(k-1)}}$ enjoys the same upper bound. Given $\eps>0$, we can therefore pick $J$ large enough that $3(k+2)e^ke^{-J(k-1)} \vee e^{k+3-J} < \eps$. Then 
\begin{align*}
    \pr{d_{GH}(\TABin, \TABn) \geq \eps} \leq \eps +\pr{\sigma_{i,n} \geq e^Jn^{\frac{k}{k+1}}}.
\end{align*}
Note that, by considering intersection times in the time interval $[\frac{1}{2}e^Jn^{\frac{k}{k+1}}, e^Jn^{\frac{k}{k+1}}]$, we have by Markov's inequality that
\[
\pr{\sigma_{i,n} \geq e^Jn^{\frac{k}{k+1}}} \leq \pr{\textsf{Binomial}\left(\frac{1}{2}e^Jn^{\frac{k}{k+1}}, (\frac{1}{2}e^Jn^{\frac{k}{k+1}}-i)^kn^{-k}\right) \geq i} \leq \frac{e^{2J}}{i}.
\]
In particular, since $J$ is already fixed we can now choose $i$ large enough that this is less than $\eps$ and we deduce that, for all large enough $i$, 
\begin{align*}
    \pr{d_{GH}(\TABin, \TABn) \geq \eps} \leq 2\eps.
\end{align*}
In particular, since $\eps>0$ was arbitrary this implies the result.
\end{proof}

\subsection{Proof of Theorem \ref{thm:AB convergence} (from GH to GHP)}

To complete the proof, we first note that it follows from Corollary \ref{cor:stick breaking conv}, Proposition \ref{prop:finite trees converge AB} and Corollary \ref{cor:AB tightness} that the convergence of Theorem \ref{thm:AB convergence} holds with respect to the GH topology. It remains to add the Prokhorov convergence. 

By \cite[Theorem 4.1]{khezeli2020metrization}, the space of boundedly compact pointed metric space is Polish, hence it follows from the Skorokhod Representation theorem that we can assume that the GH convergence of the previous subsection occurs almost surely (although we did not explicitly add a pointed vertex to our space, we could do this by setting the roots to be $\TAB_{0,n}$ and $\rho(0)$). In particular this implies that almost surely, there exists an embedding of all the spaces $(n^{-\frac{k}{k+1}}\TABk (K_n))_{n \geq 1}$ and $\Tbg$ into a common compact metric space of the form  
\[
\Tbg \sqcup \bigsqcup_{n \geq 1} n^{-\frac{k}{k+1}}\TABk (K_n)
\]
such that the spaces converge almost surely with respect to the Hausdorff metric. Moreover, by Corollary \ref{cor:AB tightness} we can assume that for all $i \geq 1$, $d_{H}(\TABin, \TABn) \to 0$, and more specifically that for all $i$ and any fixed $\eps>0$, the conditions of Proposition \ref{prop:stick breaking close} are eventually satisfied for all sufficiently large $n$.

We let $\mu_i$ denote the measure on $T^{(i)}$ considered in Proposition \ref{prop:measure conv SB}. We also let $\nu_{i,n}$ denote the restriction of the measure $\nu_n$ on $\TABk (K_n)$ to $\TABin$.

\begin{lemma}\label{lem:tightness measures AB}
Almost surely,
    \[
\lim_{i \to \infty} \limsup_{n \to \infty} d_P(\nu_{i,n}, \nu_n) = 0.
\]
\end{lemma}
\begin{proof}
By analogy with the stick-breaking limit, we fix a parameter $\gamma$ which will be equal to $k-1$ in the maximal case, and $0$ in the uniform case. Fix $i$, condition on $\TABin$ and take a subset $A \subset \TABin$. Consider the evolution of an urn model as in Section \ref{sctn:urn} starting at
\[
(U_i, V_i) = \left(\sum_{t \leq \sigma_{i,n}} t^{\gamma} \mathbbm{1}\{X_t \in A, \nexists s < t \text{ such that } X_s = X_t\}, \sum_{t \leq \sigma_{i,n}} t^{\gamma} \mathbbm{1}\{X_t \notin A, \nexists s < t \text{ such that } X_s = X_t\}\right),
\]
and note that the dynamics are exactly so that 
\[
\left(\frac{U_j}{U_j+V_j}, \frac{V_j}{U_j+V_j}\right)_{i \leq j \leq n} \overset{(d)}{=} \left(\nu_{j,n}(A^{\uparrow}), 1-\nu_{j,n}(A^{\uparrow})\right)_{i \leq j \leq n},
\]
where $A^{\uparrow}$ denotes the set of descendants of $A$ (with respect to the $\ell^1$ embedding; recall we say that $x$ is a \textbf{descendant} of $y$ if $x$ if the representative of $x$ in $\ell_1$ is the concatenation of that of $y$ and another vector). In particular, the sequence $(\nu_{j,n}(A^{\uparrow}))_{j \geq i}$ is a martingale and moreover we claim that
\begin{align}\label{eqn:measure conv sup}
    \prcond{\max_{j>i} |\nu_{j,n}(A^{\uparrow}) - \nu_{i,n}(A^{\uparrow})| \geq \delta }{\TABin}{} \leq C_{i,n} \delta^{-2},
\end{align}
where $\sup_{j \geq i} \sup_n C_{j,n} \to 0$ in probability as $i \to \infty$. To see this, we apply Doob's $L^2$ inequality, which implies that
\[
\econd{\left(\max_{i \leq j \leq n} |\nu_{j,n}(A^{\uparrow}) - \nu_{i,n}(A^{\uparrow})|\right)^2}{\TABin} \leq 4 \econd{|\nu_{n,n}(A^{\uparrow}) - \nu_{i,n}(A^{\uparrow})|^2}{\TABin} = 4 \Var \left({\nu_{n,n}(A^{\uparrow})} \middle | {\TABin} \right).
\]
By \eqref{eqn:urn variance bound}, we have that
\begin{align*}
\Var \left({\nu_{n,n}(A^{\uparrow})} \middle | {\TABin} \right) \leq \sum_{i \leq j \leq n} \econd{ \left( \frac{\Delta_{j+1}}{U_j + V_j} \right)^2}{\TABin}{} = \sum_{i \leq j \leq n} \econd{ \left( \frac{|\TABjjn|^{\gamma+1} - |\TABjn|^{\gamma+1}}{|\TABjn|^{\gamma+1}} \right)^2}{\TABin}.
\end{align*}
To bound the latter sum, note that, conditionally on $\TABjn$, the quantity $|\TABjjn| - |\TABjn|$ is dominated by a geometric law with parameter $\left(\frac{|\TABjn|}{n}\right)^k$, and hence
\begin{align*}
    \prcond{|\TABjjn|^{\gamma+1} - |\TABjn|^{\gamma+1} \geq x}{|\TABjn| }{} &\leq \prcond{|\TABjjn| \geq (|\TABjn|^{\gamma+1}+ x)^{1/(\gamma+1)}}{|\TABjn| }{} \\
    &\leq \prcond{|\TABjjn| - |\TABjn| \geq x |\TABjn|^{-\gamma}}{|\TABjn| }{} \\
    &\leq \left(1- \left(\frac{|\TABjn|}{n}\right)^k \right)^{x|\TABjn|^{-\gamma}} \\
    &\leq \exp \left\{- \left(\frac{|\TABjn|^{k-\gamma}}{n^k}\right) x\right\}.
\end{align*}
In particular, there exists $C<\infty$ so that 
\[
\econd{(|\TABjjn|^{\gamma+1} - |\TABjn|^{\gamma+1})^2}{|\TABjn|} \leq \frac{Cn^{2k}}{|\TABjn|^{2(k-\gamma)}} \qquad
\]
and hence
\[
\econd{\frac{(|\TABjjn|^{\gamma+1} - |\TABjn|^{\gamma+1})^2}{|\TABjn|^{2(\gamma+1)}}}{|\TABjn|} \leq \frac{Cn^{2k}}{|\TABjn|^{2(k+1)}} 
\]
Note that $|\TABjn|$ is typically of order $j^{\frac{1}{k+1}}n^{\frac{k}{k+1}}$, so we expect that
\[
\sum_{i \leq j \leq n} \econd{ \left( \frac{|\TABjjn|^{\gamma+1} - |\TABjn|^{\gamma+1}}{|\TABjn|^{\gamma+1}} \right)^2}{\TABin} \approx 
\sum_{i \leq j \leq n} \frac{C}{j^2}.
\]
To indeed bound the expectations, note that
\begin{align*}
\prcond{|\TABjn| < yj^{\frac{1}{k+1}}n^{\frac{k}{k+1}}}{\TABin}{} \leq \pr{\textsf{Binomial} \left(yj^{\frac{1}{k+1}}n^{\frac{k}{k+1}}, \left(\frac{yj^{\frac{1}{k+1}}n^{\frac{k}{k+1}}}{n}\right)^k\right) > j} \leq y^{3(k+1)}
\end{align*}
by Markov's inequality (with the third moment). Hence there exists $C<\infty$ so that the expectation of the $j^{th}$ term is upper bounded by $\frac{C}{j^2}$ for all $j, n \geq 1$ and we deduce that the sum tends to $0$ as $i \to \infty$, uniformly over $n$. In particular this establishes \eqref{eqn:measure conv sup}. Since we already know that the trees converge in the Hausdorff distance in the sense of \cref{cor:stick breaking conv}, \cref{prop:finite trees converge AB} and \cref{cor:AB tightness}, standard techniques then allow us to take an $\eps$-cover of $\TABin$ and apply \eqref{eqn:measure conv sup} to all sets in the cover in order to obtain the result of the lemma (we refer to \cite[around Equation (40)]{AldousCRTI} for an example).
\end{proof}

It just remains to tie everything together.

\begin{proof}[Proof of Theorem \ref{thm:AB convergence}]    
By Proposition \ref{prop:measure conv SB}, we know that $d_P(\mu_i, \mu) \to 0$ as $i \to \infty$. By Lemma \ref{lem:tightness measures AB}, we know that
\[
\lim_{i \to \infty} \limsup_{n \to \infty} d_P(\nu_{i,n}, \nu_n) = 0.
\]
Hence it is sufficient to show that, for each fixed $i \geq 1$, $d_P(\nu_{i,n}, \mu_i) \to 0$ almost surely. This is in fact a straightforward consequence of Proposition \ref{prop:Prokh conv measure on sticks} and Proposition \ref{prop:stick breaking close}.
\end{proof}

 \section{Equivalence of choice trees: proof of \cref{thm:choice trees the same}} \label{sec: wilson trees}

Recall the definition of Wilson's algorithm with choice as defined in \cref{sec: Wilson algorithm trees}. By ignoring steps at which $v_{i+1}$ are already in $T_i$, this algorithm allows us to sample the partial Wilson choice tree step by step, such that the number of leaves grows by one within each step. We will compare each such step to adding a new branch in the \AB algorithm and show that given that the trees built until this step are of the same size, the distribution of the length of every such branch is the same in both algorithms.

That is, fix $n$ and write $T_j$ for the partial tree constructed after running $j$ steps with non-zero branch lengths of Wilson's algorithm with choice (zero branch length can happen only if we start our choice loop erased random walk from a vertex $v$ which is already in the tree). We will write $Y$ for the maximal choice loop-erased random walk used in the construction (i.e. as in \cref{sec: Wilson algorithm trees}), started at some $v \notin T_j$ and terminated when hitting $T_j$; we write $\tau$ for the first time the walk hits $T_j$. It follows by construction that $\tau$ is distributed exactly as the length of the $(j+1)^{th}$ (non-zero length) branch of the tree.

To prove \cref{thm:choice trees the same}, we claim that it is sufficient to prove the following proposition (we assume that $k \geq 1$ is already fixed).

\begin{proposition}
\label{prop:main stick breaking ingredient Wilson}
For every $n\in \NN$ and $i, j, m \leq n$ we have that
\[
\prcond{\tau > i+1}{\tau>i, |T_j| = m}{} = 1 - \left( \frac{ m+ i+1}{n} \right)^k.
\]

\end{proposition}


Indeed, it is straightforward to verify that, if $\hat{T}_j$ and $\hat{\tau}$ denote the analogous quantities in the \AB algorithm (i.e. only looking at non-empty branches), then $\prcond{\hat{\tau} > i+1}{\hat{\tau}>i, |\hat{T}_j| = m}{}$ is also given by the above expression. This enables us to conclude that the \AB choice trees and the Wilson choice trees have the same distribution of branch lengths in each of the partial spanning trees. Hence, we can couple the two processes so that the length of all branches are exactly the same. Furthermore, in both the maximal and uniform cases, the attachment laws are defined to be exactly the same. Hence we can conclude that the \AB choice tree and the Wilson choice tree have the same distribution for every $n\in \NN$.

Therefore, given \cref{prop:main stick breaking ingredient Wilson}, the result of \cref{thm:choice trees the same} is immediate. We will thus now give the proof of \cref{prop:main stick breaking ingredient Wilson}.

 \subsection{Laplacian random walk viewpoint}

For every $k$, in order to sample the Wilson $k$-choice trees, we have to analyze the corresponding loop-erased maximal choice random walk from some vertex $v$ to a predetermined set. In \cref{sec: lerw to rayleigh}, we analyzed loop erased choice random walks running forever, and proved that they are typically of length $\nkk$ (note however the slight distinction that in that case we took $\Ter = \emptyset$, whereas in Wilson's algorithm we take $\Ter = T_{j-1}$ when sampling the $j^{th}$ branch; in both cases we take $\Av_n = \LE (X[0,n])$, however).

We will prove \cref{prop:main stick breaking ingredient Wilson} using the \textit{Laplacian random walk representation.} 

By considering the last time that a random walk visits a specific vertex $v$ and following the same proof as in the classical case (see \cite[Exercise 4.1]{LyonsPeres}), it follows that the formula \eqref{eqn:Laplacian RW classical formula} is also valid for maximal choice random walks, using the fact that if $Y_0 , \ldots , Y_{k-1}$ is a prefix of a choice random walk, the probability to hit it does not change as the walk evolves (in contrast to uniform choice random walk, where the probability grows as the trace of the walk grows).

To ease calculations, we will consider the walk on the graph $K_n$ \textbf{with self loops}. We apply this at step $j+1$ of the algorithm; for this we let $Y$ denote the loop-erasure of a \textit{maximal} choice random walk $X$ run from a vertex $v$ to $T_{j}$, let $\tau^Y$ be the time at which $Y$ hits $T_{j}$, and let $\tau_A$ denote the hitting time of the set $A$ for the choice random walk $X$. Then \eqref{eqn:Laplacian RW classical formula} implies that
\begin{align*}
\prcond{\tau^Y = i+1}{\tau^Y>i}{} = \prcond{ Y_{i+1} \in T_{j}}{(Y_m)_{m=0}^{i}}{v} &= \prcond{ X_{1} \in T_j}{\tau_{T_j} < \tau^+_{\cup_{m=0}^{i} \{Y_m\}}}{Y_i} = \frac{\prstart{ X_{1} \in T_{j}}{Y_i}}{\prstart{\tau_{T_j} < \tau^+_{\cup_{m=0}^{i} \{Y_m\}}}{Y_i}}. 
\end{align*}
Now note that, at each step of the $k$-choice maximal random walk starting from $Y_i$ used in Wilson's algorithm, the probability of hitting $T_j$ on the next step is $\left( \frac{|T_j|}{n} \right)^k$, and the probability of hitting the union $T_j \bigcup \cup_{m=0}^{i} \{Y_m\}$ is $\left( \frac{|T_j\bigcup \cup_{m=0}^{i} \{Y_m\}|}{n} \right)^k$. Hence, for a \textbf{simple random walk} we have that 
\[
\prstart{\tau_{T_j} < \tau^+_{\cup_{m=0}^{i} \{Y_m\}}}{Y_i} = \left( \frac{|T_j|}{|T_j|+i+1} \right)^k, \qquad \prstart{ X_{1} \in T_{j}}{Y_i} = \left( \frac{|T_j|}{n} \right)^k,
\]
and substituting back we see that 
\[
\prcond{\tau^Y > i+1}{\tau^Y>i}{} = 1 - \left( \frac{|T_j|+ i+1}{n} \right)^k,
\]
as in \cref{prop:main stick breaking ingredient Wilson}.

We end the section with some simulations of (uniform) choice spanning trees.

\begin{figure}[h]
  \subfigure[UST ($1$-choice tree)]{\includegraphics[width=0.33\textwidth]{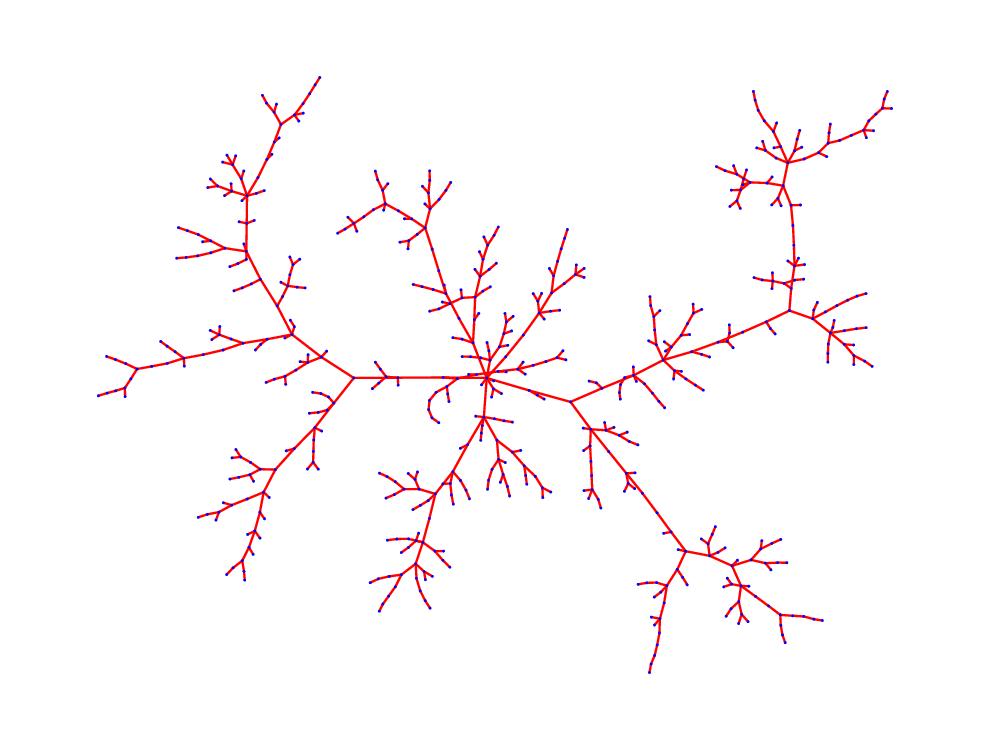}}
  \subfigure[$2$-choice tree]{\includegraphics[width=0.33\textwidth]{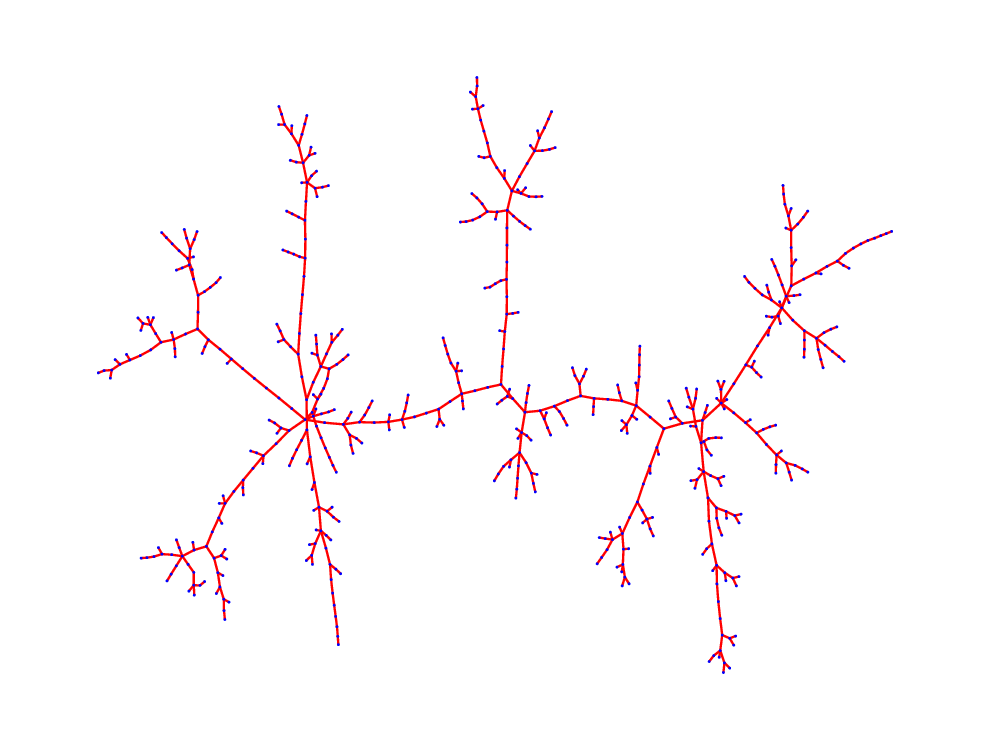}}
  \subfigure[$3$-choice tree]{\includegraphics[width=0.33\textwidth]{3choice.eps}}

  \subfigure[$4$-choice tree]{\includegraphics[width=0.5\textwidth]{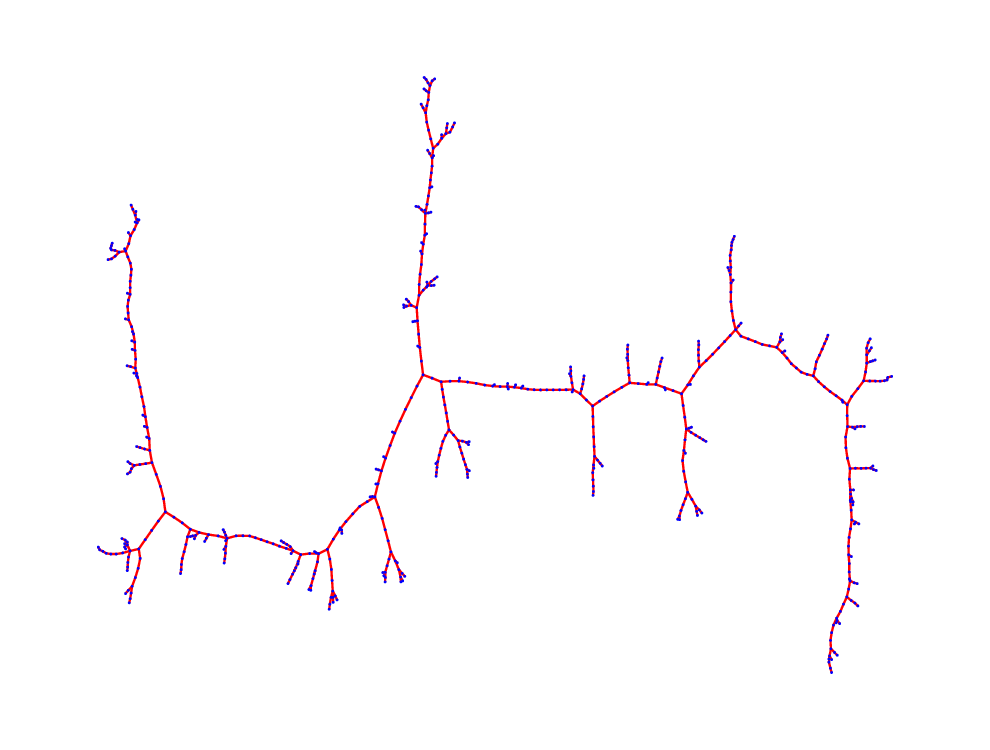}}
  \subfigure[$5$-choice tree]{\includegraphics[width=0.5\textwidth]{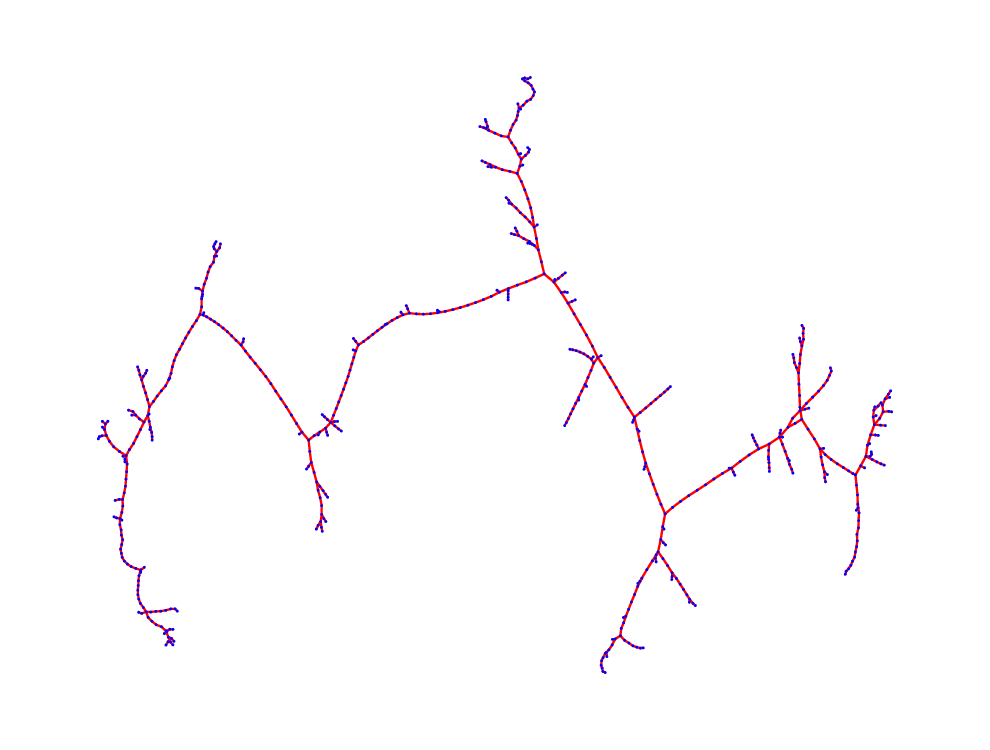}}
  \caption{Some choice trees}
  \label{fig:row1}
\end{figure}

 




\section{Open questions}\label{sctn:open questions}

In this question we list some open problems that follow on naturally from this work. We plan to address some of these in future papers.




In a first direction, an interesting generalization of the results in this paper concern the \AB chain. In the case of USTs, the \AB chain is obtained by running a random walk on $K_n$ beyond its cover time and letting $X_n(t)$ denote the set of last exit edges of the random walk up until time $t$. If $t \geq \tau_{\text{cov}}$, then $X_n(t)$ is a spanning tree and $X_{n+1}(t)$ is obtained from $X_n(t)$ by deleting an edge to leave two subtrees, and then re-attaching the two subtrees in a new location. The UST is the stationary measure for this Markov chain in the sense that if we consider the chain via $Y_n(t) = X_n(\tau_{\text{cov}}+t)$, then $Y_n(t)$ is distributed as $\UST(K_n)$ for all $t \geq 0$. It was moreover verified by Evans, Pitman and Winter \cite{EPW} that the whole chain $(Y_n(t))_{t \geq 0}$ converges in the scaling limit to an analogous chain on continuum trees know as \textit{root growth with re-grafting}, for which the stationary measure is the CRT.

A natural question is whether we can also take scaling limits of an analogous \AB chain constructed using choice random walks. In the case of the uniform attachment rule, the trees in fact evolve according the same rule as in the classical UST case after the cover time and it therefore follows that the whole chain again converges to the same root growth with re-grafting process of \cite{EPW}. However, we expect to see different behaviour for the maximal attachment rule.

\begin{question}
Let $Y_n(t)$ denote the set of last exit edges of a maximal $k$-choice random walk on $K_n$ at time $\tau_{\text{cov}}+t$. Does the chain $(n^{-\frac{k}{k+1}}Y_n(tn^{\frac{k}{k+1}}))_{t \geq 0}$ have a scaling limit, and can this be constructed via a root growth re-grafting type process applied to the tree $\T_{k, k-1}$? Is the maximal $k$-choice \AB tree stationary for the discrete process?
\end{question}
The first step in this direction (which should be straightforward) would simply be to verify whether we can replace the set of first entry edges in the $k$-choice \AB algorithm with the set of last exit edges, and still obtain a tree with the same law. 

In another direction, it would be interesting to consider \textit{local limits} rather than scaling limits. The local limit of $\UST(K_n)$ is well-known to be a \textsf{Poisson}($1$) Galton--Watson tree, conditioned to survive \cite{Grimmett_1980}.

\begin{question}
    What are the local limits of $\TABk_{\text{max}}(K_n)$ and $\TABk_{\text{unif}}(K_n)$?
\end{question}



Another question refers to \textit{universality} of these scaling limits. For classical USTs, the result that $\UST(K_n)$ rescales to the CRT is actually a special case of a much more general result: the CRT is also the scaling limit of $\UST (\mathbb{T}^d_n)$ for $d \geq 5$, where $\mathbb{T}^d_n$ is the $d$-dimensional torus with side length $\lfloor n^{1/d} \rfloor$, and also other high dimensional graphs such as he hypercube and expander graphs (see \cite{PeresRevelleUSTCRT}, \cite{ANS2021ghp} and\cite{archer2023ghp} for more precise results). This leads us to the following question.

\begin{question}
    Do the scaling limit results of \cref{thm:AB convergence} also hold for choice spanning trees of the high-dimensional torus, and for other high-dimensional graphs?
\end{question}

Note that, to define $\TWilk (\mathbb{T}^d_n)$, or on another general graph, one must first specify an ordering of the vertices, since this order a priori affects the outcome of the sampling algorithm. A natural choice would be to take a uniform ordering.

We remark also that we do not expect \cref{thm:choice trees the same} to hold for other graphs (except perhaps in some very special cases); the previous question could therefore be explored for either of the two algorithms.

Finally, we remark that USTs satisfy some nice properties linking them with determinants of certain matrices; more precisely the probability that a certain set of edges appear in $\UST(G)$ is given by the determinant of a matrix related to $G$ (this is known as the transfer-current theorem). 

\begin{question}
    Given a set of edges $(e_1, \ldots, e_m) \subset K_n$, can the probability $\pr{(e_1, \ldots, e_m) \in \TABk_{\text{max}}(K_n)}$ be expressed as the determinant of a related matrix? Similarly for $\TABk_{\text{unif}}(K_n)$?
\end{question}
\clearpage

\bibliographystyle{abbrv}
\bibliography{biblio}

\end{document}